\documentclass{article}
\usepackage{graphicx, float, cite}
\usepackage{amssymb}  
\usepackage{amsmath}
\usepackage{amsthm}

\title{Extending wavelet regularity beyond Gevrey classes}
\author{Filip Tomić, Stefan Tuti\'c, Milica \v Zigi\'c}
\date{}

\allowdisplaybreaks 

\def\be{\begin{equation}}
\def\ee{\end{equation}}
\def\t{\tau}
\def\s{\sigma}
\def\dss{\displaystyle}

\newcommand{\ctg}{\operatorname{ctg}}
 \def\E{\mathcal{E}}

\newtheorem{te}{Theorem}[section]
\newtheorem{lema}[te]{Lemma}
\newtheorem{prop}[te]{Proposition}

\newtheorem{rem}[te]{Remark}
\newtheorem{ex}[te]{Example}
\newtheorem{de}[te]{Definition}
\newtheorem*{theoremnonumber}{Theorem}

\begin{document}
\maketitle
\begin{abstract}
We present a general construction of smooth orthonormal wavelet $\psi$ which, together with its Fourier transform $\widehat{\psi}$ belongs to the extended Gevrey class $\mathcal{E}_{\sigma}(\mathbb{R})$ for $\sigma > 1$, providing an example that lies beyond all classical Gevrey classes. Our approach uses the idea of invariant cycles to extend the initial Lemari\'e-Meyer support of the low-pass filter $m_0$ from $ [-\frac{2\pi}{3}, \frac{2\pi}{3}]$ to $ [-\frac{4\pi}{5}, \frac{4\pi}{5}]$. This extension allows precise control of the the decay rate of $m_0$ near $\frac{2\pi}{3}$, which yields global decay estimates for $\psi$ and $\hat\psi$. In addition, the decay rates are described using special functions involving the Lambert $W$ function, which plays an important role in our construction.
\end{abstract}


\section{Introduction}


In this paper, we present a construction of smooth, non-band-limited  orthonormal wavelets $\psi$ whose regularity is weaker than Gevrey regularity both in time and in frequency. This is achieved by controlling the decay of the low-pass filter $m_0$
near the invariant cycle points $\pm 2\pi/3$ \cite{Fukuda}. For this purpose, we use certain flat functions and their extensions to complex domain, see \cite{Javier01, Javier02, Javier03}.
Our methods differ from those used in \cite{TT-02}
where band-limited  wavelets were constructed instead.

\par

The notion of a \emph{multiresolution analysis} (MRA) is one of the fundamental concepts in wavelet theory, introduced by Y.~Meyer~\cite{Meyer} and later developed by S.~Mallat~\cite{Mallat}. MRA is based on a hierarchical family of closed subspaces of $L^2(\mathbb R)$ and involves a \emph{scaling function} $\varphi$, which plays a central role in construction of the orthonormal wavelet bases. Under certain technical assumptions, its Fourier transform $\hat\varphi$ can be represented as an infinite product constructed from a $2\pi$-periodic function, known as the low-pass filter. More precisely, if $\varphi$ is a scaling function and $m_0$ the associated low-pass filter, then  
\begin{equation}
\label{OdnosM0iPhi}
\hat{\varphi}(\xi)=\prod_{j=1}^{\infty} m_0(2^{-j}\xi), \qquad \xi\in\mathbb R,
\end{equation}
and the Fourier transform of the corresponding MRA wavelet $\psi$ is given by  
\begin{equation}
\label{MRAWavelet}
\hat{\psi}(\xi)= e^{i \xi/2}\, \overline{m_0\!\left(\frac{\xi}{2}+\pi\right)}\, \hat{\varphi}\!\left(\frac{\xi}{2}\right), \qquad \xi\in\mathbb R.
\end{equation} In such framework, it is therefore sufficient to construct a low-pass filter with the desired regularity properties, which are then inherited by $\hat{\varphi}$ and $\hat{\psi}$ through \eqref{OdnosM0iPhi} and \eqref{MRAWavelet}. The regularity of $\varphi$ and $\psi$ is then given by Paley--Wiener type theorems, which relate the smoothness of a function to the decay of its Fourier transform. For a more detailed exposition of the general theory of wavelets and their applications, we refer to classical textbooks \cite{Dau, HW, MeyerBook}.

We shall use invariant cycles $\left\{ -\frac{2\pi}{3}, \frac{2\pi}{3}\right\}$ and $\left\{ -\frac{2\pi}{5}, \frac{2\pi}{5}, - \frac{4\pi}{5}, \frac{4\pi}{5}\right\}$ to extend the support of the low-pass filter $m_0$ on the interval $[-\pi,\pi)$. For this purpose, let us recall the definition of an invariant cycle (see \cite{HW1}).


\begin{de}
    Let $\rho: [-\pi,\pi)\to [-\pi,\pi): \xi\mapsto2\xi \ (mod \, 2\pi$). For $\xi\in [-\pi,\pi)$, if there exists $l\in\mathbb{N}$ such that $\rho^{l} (\xi) = \xi$, we call the corresponding orbit $\mathcal{O}(\xi)=\{ \xi,\rho(\xi),\dots, \rho^{l-1} (\xi) \}$ an invariant cycle of length $l$, generated by $\xi$.
\end{de} 
The notion of invariant cycles was introduced by A.~Cohen to characterize orthonormality conditions for wavelets (see \cite{Dau}). 
In \cite{HW1}, it is proved that for a band-limited MRA wavelet, 
$\prod_{\xi\in \mathcal{O}(\xi)} m_0(\xi) = 0$ for every invariant cycle $\mathcal{O}(\xi)$, and furthermore the authors constructed a smooth low-pass filter with $m_0(\pm 2\pi/3)\neq 0$, 
yielding a non-band-limited wavelet regular only up to a finite order. 
In contrast, Lemarié–Meyer (see \cite{HW}) use smooth low-pass filter that vanishes in a neighborhood of $\pm 2\pi/3$, making the corresponding scaling function $\varphi$ 
and wavelet $\psi$ band-limited and thus analytic in time. In our construction,
$m_0(\pm 2\pi/3) = 0$ and $m_0(\xi) \neq 0$ in a neighborhood of $\pm 2\pi/3$
($\xi \neq \pm 2\pi/3$), yielding to non-band-limited  
scaling function $\varphi$ and wavelet $\psi$ with extended Gevrey regularity.





It is well known that a smooth orthonormal wavelet cannot have exponential decay
(see \cite{DH}). Therefore, constructing a smooth wavelet  whose decay is slower than exponential appears to be a challenging task (see \cite{Hernandez, Moritoh}). An example of such a wavelet can be found in \cite{Fukuda}, where the authors employ Gevrey regularity and use the translates of the Gevrey function $e^{-|\xi|^{-\frac{1}{s-1}}}$, $s>1$, to control the decay rate of the low-pass filter near $\frac{2\pi}{3}$.

We employ the ideas from \cite{Fukuda} and consider functions of the form  
\begin{align}
\label{BazicnaFunkcija}
\begin{split}
&f_{\rho , \sigma} (x) = e^{- \rho\, g_{\sigma} \left(1/x\right)}, \qquad  \rho>0,\; \sigma>1, \; x\not=0,\\
&\text{where}\\
&g_{\sigma}(x) = \omega_{\sigma}\!\left( \ln \left(1+|x|\right) \right),\\
&\omega_{\sigma}(x) = x^{\frac{\sigma}{\sigma - 1}} \big/ W^{\frac{1}{\sigma - 1}}(x),\quad \omega_{\sigma}(0)=0,
\end{split}
\end{align}
where $W$ denotes the principal branch of the Lambert $W$ function (see subsection \ref{subsec:preliminaries}). This approach yields a new class of non-band-limited orthonormal wavelets whose decay rate is faster than polynomial but slower than subexponential. Moreover, its regularity in time and frequency domains is captured by the \emph{extended Gevrey classes}.

Extended Gevrey classes of locally smooth functions were introduced in \cite{PTT-01}. They contain the union of all Gevrey classes and therefore describe an intermediate regularity between Gevrey and $C^{\infty}$. The growth of the derivatives of their elements is controlled by the sequence $\{M^{\tau,\sigma}_p\}_{p \in \mathbb{N}_0}$, where
\begin{equation}
\label{DefinicijaNiza}
M^{\tau,\sigma}_p = p^{\tau p^{\sigma}},\quad p \in \mathbb{N},\qquad 
M_0^{\tau,\sigma} = 1,\qquad 
\tau > 0,\ \sigma > 1.
\end{equation}

Since the sequences in \eqref{DefinicijaNiza} do not satisfy Komatsu's condition 
$$\displaystyle (M.2)\quad (\exists\, C>0)\ M_{p+q} \leq C^{p+q+1} M_p M_q,\quad p,q\in \mathbb N,$$ for any choice of $\tau>0$ and $\sigma>1$, the extended Gevrey classes cannot be treated within the classical framework of ultradifferentiable function theory. Nevertheless, it has recently been shown that extended Gevrey classes form an important example of \emph{weight matrix classes}, where regularity is governed by a family of defining sequences (see \cite{Javier01, TT-01}). Applications of the theory can be found in \cite{CL, Javier02, Javier03, Lastra}, where extended Gevrey classes are referred to as PTT-spaces.


The main purpose of this paper is to establish new regularity properties of orthonormal wavelets that go beyond the classical theory of ultradifferentiable functions. We refine the result of \cite{TT-02} on band-limited wavelets in extended Gevrey settings by controlling the decay and smoothness of the scaling function in the Fourier domain. At the same time, we obtain a wavelet that is less regular than the one constructed in \cite[Theorem~4.4]{Fukuda}; equivalently, its regularity lies closer to $C^{\infty}$ (see Lemma~\ref{PoslednjaLema}).

Paper is organized as follows: In Subsection \ref{subsec:preliminaries} we fix the notation and discuss the main properties of the Lambert $W$ function. In Section~\ref{subsec:regularityclasses}, we introduce the extended Gevrey classes $\E_{\sigma}(\mathbb R)$, see Definition \ref{definicijaKlase}, discuss their defining sequences \( M^{\tau,\sigma}_p \) and corresponding associated functions. In particular, we prove that functions $f_{\rho,\sigma}$ given in \eqref{BazicnaFunkcija} belong to $\E_{\sigma}(\mathbb R)$, see Theorem \ref{PrimerExtGevrey} and Proposition \ref{proizvoljnorho}. In Section \ref{section:waveletconstruction} we construct the desired wavelet from the low-pass filter $m_0$, by using $f_{\rho,\sigma}$, and prove Theorem \ref{MainResult} as our main result. This implies the following particular result:

\begin{theoremnonumber}
There exists an orthonormal wavelet $\psi$ such that for arbitrary $\eta>1$
$$\displaystyle \psi\in \E_{2}(\mathbb R)\backslash \bigcup_{1<\sigma'<\frac{2+\eta}{1+\eta}}\E_{\sigma'}(\mathbb R) \quad{\rm and}\quad \displaystyle \hat\psi\in \E_{2}(\mathbb R)\backslash \bigcup_{1<\sigma'<2}\E_{\sigma'}(\mathbb R).$$ 
\end{theoremnonumber} 

To illustrate our construction, we present several graphs of ONW given by Theorem \ref{MainResult} in Subsection \ref{secIlustracije}.

\subsection{Preliminaries}
\label{subsec:preliminaries}

Throughout the paper we use the following notation: $\mathbb{N}$, $\mathbb{N}_0$, $\mathbb{R}_+$, $\mathbb{R}$ and $\mathbb{C}$
denote the sets of natural numbers, non-negative integers, positive real numbers,
real numbers, and complex numbers, respectively. We write $f \asymp g$ to denote that the two functions are
\emph{asymptotically equivalent}, meaning that $f = O(g)$ and $g = O(f)$
as $x \to \infty$ (here $f = O(g)$, $x \to \infty$, means that
$f(x) \leq L\,(g(x) + 1)$ for some $L \geq 1$ and all $x \geq 0$). By $f \prec g$ we mean that $f = o(g)$ as $x \to \infty$, that is,
$f(x)/g(x) \to 0$ as $x \to \infty$. We also write $f\sim g$ when $f(x)/g(x) \to 1$ as $x \to \infty$.

With $\displaystyle{\rm coz} \,f=\{x_0\in \mathbb R\,|\, f(x_0)\not=0\}$ we denote a \emph{cozero} set (complement of a zero set) of the continuous function $f$. Then the \emph{support} of $f$ is ${\rm supp}\, f=\overline{{\rm \operatorname{coz}} \,f}$, where $\overline{X}$ denotes the closure of the set $X$. The \emph{interior} of the set $X$ is denoted by ${\rm int}\,X$. We write $\mathbb K\subset\subset  U$ when $\mathbb K$ is a compact subset of an open set $U$.


Let $\nu$ be a non-negative, continuous, increasing and even function on $\mathbb R$, with $\nu(0)=0$.
Then (see \cite{BMT}), $\nu$ is of \emph{Braun-Meise-Taylor} type (BMT in short) if the following conditions hold:
\newline
($\alpha$) \hspace{1em} $\displaystyle \nu(2x)=O(\nu(x)),\quad x\to\infty,$
\newline
($\beta$) \hspace{1em} $\displaystyle \nu(x)=O(x),\quad x\to\infty,$
\newline
($\gamma$) \hspace{1em} $\displaystyle \ln x = o(\nu(x)), \quad x \to \infty,\;\;$ 
\newline
($\delta$) \hspace{1em} $\displaystyle \vartheta(x)=\nu (e^x)\quad$ is convex.

\begin{ex}
\label{exampleBMT}
Some classical examples of BMT functions are
\begin{align*}
\nu (x)=|x|^s,\,0<s\leq 1,\quad\nu (x) = \ln^{s}_+ |x|,\, s>1,\quad \nu(x)=\frac{|x|}{\ln^{s-1} (e+|x|)},\,s>1,  
\end{align*}
for $x\in \mathbb R$ where $\ln_+{|x|}=\max\{0, \ln|x|\}$.
Moreover, it was shown in \cite[Theorem 1]{TT-01} that $\omega_{\sigma}(\ln_+|x|)$, $x\in\mathbb R$, where $\omega_{\sigma}$ is given in \eqref{BazicnaFunkcija}, is asymptotically equivalent to a BMT function. 
\end{ex}


Next, we recall some of the basic properties of the \emph{Lambert $W$ function}. It is defined as the multivalued inverse of the function \( z \mapsto z e^{z} \) for \( z \in \mathbb{C} \). We will denote its principal branch by \( W(z) \), \( z \in \mathbb{C} \setminus (-\infty, -e^{-1}] \), which is also denoted by $W_0(z)$ in the literature. Lambert $W$ function splits the complex plane into infinitely many regions. The boundary curve of its principal branch is given by: 
$$\{(-x\, \ctg x,x)\in \mathbb R^2\,|\,-\pi <x<\pi\}.$$ 
Here we list some additional properties of $W$ that we will use in the sequel:
\begin{itemize}
\item[$(W1) \quad$] $W(-e^{-1})=-1$, $W(0)=0$, $W(e)=1$ and $W(x)$ is continuous, increasing and concave on $(-e^{-1},\infty)$.
\vspace{0.2cm}

\item[$(W2) \quad$] $ z=W(z)e^{W(z)}$,  \( z \in \mathbb{C} \setminus (-\infty, -e^{-1}] \).

\item[$(W3) \quad $] $\displaystyle\ln x -\ln(\ln x)\leq W(x)\leq \ln x-\frac{1}{2}\ln (\ln x), \quad x\geq e.$


\item[$(W4) \quad$] Derivatives of Lambert $W$ function are given by
\begin{align*}
    W^{(n)} (x) = \frac{W^n (x)\, p_n (W(x))}{x^n (1+W(x))^{2n-1}}, \quad x>0,\quad n\in\mathbb{N},
\end{align*} where $p_n$ is polynomial satisfying 
$$p_{n+1}(x) = (1+x)p'_n (x) - (nx+3n-1)p_n (x), \quad p_1 (x)=1.$$ In particular,

\begin{align}
\label{prviizvodlambert}
W'(x) = \frac{W(x)}{x (1+W(x))} , \quad x>0.
\end{align}

\end{itemize} Note that $(W1)$ implies that $W(x)>0$ when $x>0$. From $(W3)$ it follows
\begin{equation} \label{PosledicaLambert1.5}
W(x)\sim \ln x, \quad x\to \infty,
\end{equation}
and hence $W(C x)\sim W(x),$ $ x\to \infty$, for any $C>0$. Moreover, by $(W1)$ and $(W2)$ we have
\begin{equation}
\label{LambertU0}
    W(x) \sim x ,\quad x\to 0.
\end{equation} For more details concerning the Lambert $W$ function we refer to \cite{LambF, Mezo}.

We end this preliminary section with the following lemma that describes the behavior of $e^{W(x)}$.

\begin{lema}
\label{LemaLambert}
If $W$ is the principal branch of Lambert $W$ function then 
$$\ln x \prec e^{W(x)} \prec x ,\quad x\to \infty.$$
\end{lema}
\begin{proof}
 By using L'H\^opital's rule, \eqref{prviizvodlambert} and \eqref{PosledicaLambert1.5}, we have
\begin{align*}
    \lim_{x\to +\infty} \frac{\ln x}{e^{W(x)}}
    = \lim_{x\to +\infty} \frac{\frac{1}{x}}{e^{W(x)} \frac{W(x)}{x(1+W(x))}}
    = \lim_{x\to +\infty} \frac{1}{e^{W(x)}} \frac{1+W(x)}{W(x)}
    =0.
\end{align*}
In addition, we use $(W2)$ and \eqref{PosledicaLambert1.5} to obtain
\begin{equation*}
    \lim_{x\to + \infty} \frac{e^{W(x)}}{x}
    = \lim_{x\to + \infty} \frac{e^{W(x)}}{W(x)e^{W(x)}}
    = \lim_{x\to + \infty} \frac{1}{W(x)}
    = 0.
\qedhere
\end{equation*}
\end{proof}

\section{Regularity classes}
\label{subsec:regularityclasses}

We begin with a lemma that summarizes the main properties of the sequences \( M^{\tau,\sigma}_p \), $p\in\mathbb{N}_0,$ given in \eqref{DefinicijaNiza}. We refer to \cite{PTT-01} for the proof, see also \cite{TT-03} 


\begin{lema}
\label{osobineM_p_s}
Let $\tau>0$, $\sigma>1$, $M_0 ^{\tau,\sigma}=1$,
and $M_p ^{\tau,\sigma}=p^{\tau p^{\sigma}}$, $p\in \mathbb N$.
Then there exists constant $C>1$ such that:
\begin{itemize}
\item[$(M.1)$] 
\hspace{1em} 
$ (M_p^{\tau,\sigma})^2\leq M_{p-1}^{\tau,\sigma}M_{p+1}^{\tau,\sigma}, \quad$ $p\in \mathbb N, $ 
\item[$\widetilde{(M.2)}$]  \hspace{1em}   
$ M_{p+q}^{\tau,\sigma}\leq C^{p^{\sigma} + q^{\sigma}}
M_p^{\tau 2^{\sigma-1},\sigma}M_q^{\tau 2^{\sigma-1},\sigma},\quad$
$ p,q\in \mathbb N_0,\quad$
\medskip
\item[$\widetilde{(M.2)'}$] \hspace{1em}
$ M_{p+1}^{\tau,\sigma}\leq C ^{p^{\sigma}} M_p^{\tau,\sigma}, \quad$ 
$ p\in \mathbb N_0, \quad$ 
\item[$(M.3)'$] \hspace{1em}
$ \displaystyle   \sum\limits_{p=1}^{\infty}\frac{M_{p-1}^{\tau,\sigma}}{M_p^{\tau,\sigma}} <\infty.
$
\end{itemize} 

In addition, if $\s_2>\s_1>1$ and $\t_0>0$ then for every $h,\tau>0$ there exists $C>0$ such that
\begin{equation}
\label{M4}
h^{p^{\s_1}}M^{\t_0,\s_1}_p\leq C M_p^{\tau,\s_2}.
\end{equation}
\end{lema} 

\begin{rem}
Let us briefly comment the case $\sigma=1$. Note that $\displaystyle M^{\tau,1}_p = p^{\tau p}$, $p\in\mathbb N$, for $\tau>0$ are Gevrey sequences. Then the conditions $\widetilde{(M.2)'}$ and $\widetilde{(M.2)}$ reduce to the classical Komatsu conditions  ${(M.2)'}$ and ${(M.2)}$, respectively. Moreover, the Gevrey sequences satisfy \emph{non-quasianalyticity} condition ${(M.3)'}$ if and only if $\tau > 1$. 
\end{rem}

Now we can define the extended Gevrey classes $\E_{\tau,\sigma}(\mathbb R)$ for $\tau>0$ and $\sigma>1$.

\begin{de}
\label{definicijaKlase}
Let  $\tau>0$ and $\sigma>1$ and $M^{\tau,\sigma}_p=p^{\tau p^{\sigma}}$ for $p\in \mathbb N$, $M^{\tau,\sigma}_0=1$. A smooth function $\phi$ belongs to $\E_{\tau,\sigma}(\mathbb R)$ if 
\begin{equation}
\label{DerivativeEstimates}
(\forall\, \mathbb K\subset\subset \mathbb R)(\exists C>0)\quad \sup_{x\in \mathbb K}|\phi^{(p)}(x)|\leq C^{p^{\sigma}+1} M^{\tau,\sigma}_p, \,\, p \in \mathbb N_0.
\end{equation}
We denote
\begin{align*} 
\displaystyle \E_{\sigma}(\mathbb R)=\bigcup_{\tau>0} \E_{\tau,\sigma}(\mathbb R).
\end{align*}
\end{de}

\begin{rem} In the definition of $\mathcal{E}_{\tau,\sigma}(\mathbb{R})$ the space $\mathbb{R}$ may be replaced by any open set $U\subseteq\mathbb{R}$. This produces the local classes $\mathcal{E}_{\tau,\sigma}(U)$; their elements satisfy derivative estimates \eqref{DerivativeEstimates} on every compact set $\mathbb{K}\subset\subset U$.

Again for $\sigma=1$ we obtain some of the well-known classes. For instance, \( \mathcal{E}_{1,1}(\mathbb{R}) = \mathcal{A}(\mathbb{R}) \), and \( \mathcal{E}_{\tau,1}(\mathbb{R}) = \mathcal{G}_\tau(\mathbb{R}) \) for \( \tau > 1 \) are the spaces of locally analytic functions and Gevrey functions of order \( \tau \), respectively. Recall, $\phi\in \mathcal G_{\t}(\mathbb R)$ if 

$$(\forall\, \mathbb K\subset\subset \mathbb R)(\exists C>0)\quad \sup_{x\in \mathbb K}|\phi^{(p)}(x)|\leq C^{p+1} p!^{\t}, \,\, p \in \mathbb N_0.$$
\end{rem}

Extended Gevrey classes $\mathcal{E}_{\tau,\sigma}$ for $\sigma > 1$ and $\tau > 0$ were introduced and studied as spaces of locally smooth functions equipped with projective and inductive limit topologies. Since these topologies will not be used in this paper, we omit the details and refer the reader to \cite{TT-03}. Unlike the classical Gevrey classes, note that in the definition of $\mathcal{E}_{\tau,\sigma}$ the geometric factor $C^p$ is replaced by $C^{p^{\sigma}}$, $\sigma > 1$. This modification ensures the stability of the classes under the action of differential operators.

Taking the union with respect to the parameter $\tau$, we obtain the \emph{weight-matrix class} $\mathcal{E}_{\sigma}$ of ultradifferentiable functions, whose derivatives are controlled by the family of sequences $\{ M^{\tau,\sigma}_p \}_{\tau > 0,\, p \in \mathbb{N}_0}$ for $\sigma > 1$. These classes can also be treated within the framework of BMT (Braun--Meise--Taylor) theory (see \cite{TT-01, Javier01}).


We summarize several basic properties of the extended Gevrey classes in the following Proposition (see \cite{PTT-01, PTT-02, PTT-03, TT-01, TT-02, TT-03}); the proof relies on the properties of the sequences $M^{\tau,\sigma}_p$ stated in Lemma \ref{osobineM_p_s}.


\begin{prop} 
\label{OsobineKlasa}
Let $\tau>0$ and $\sigma>1$. 
\begin{itemize} 
\item[$a)$] For $\s_2>\s_1>1$ we have
\begin{equation}
\mathcal{A}(\mathbb R)\subset\bigcup_{\tau>1}{\mathcal G}_\tau ({\mathbb R})\subset\bigcap_{\tau>0} {\mathcal{E}}_{\tau,\s_1}({\mathbb R})\subset  {\mathcal{E}}_{\tau,\s_1}({\mathbb R})\subset {\mathcal{E}}_{\s_1}\subset \E_{\s_2}\subset C^{\infty} ({\mathbb R}). \nonumber
\end{equation} 
\item[$b)$] there exists a compactly supported function $\varphi\in {\mathcal{E}}_{\tau,\sigma}({\mathbb R})$. 
\item[$c)$] ${\mathcal{E}}_{\tau,\sigma}({\mathbb R}) $ is closed under the {pointwise} multiplication.
\item[$d)$] ${\mathcal{E}}_{\tau,\sigma}({\mathbb R}) $ is closed under finite order derivation.
\item[$e)$]  $\E_{\tau,\sigma}({\mathbb R})$ is closed under superposition. In particular, if $F(x)\in \mathcal A(\mathbb R)$  and $f(x)\in \E_{\tau,\sigma}({\mathbb R})$ then $F(f(x))\in \E_{\tau,\sigma}({\mathbb R})$.
\item[$f)$] $\E_{\tau,\sigma}({\mathbb R})$ is invariant under translations and dilatations.
\end{itemize}
\end{prop}

Next we define a function associated to the sequence $M^{\tau,\sigma}_p$, $p\in \mathbb N_0$.

\begin{de} Let $\tau>0$ and $\sigma>1$. Then the associated function to the sequence $M^{\tau,\sigma}_p$, $p\in\mathbb N_0$, is given by 
\begin{align*}
\displaystyle T_{\tau,\sigma}(x)=\sup_{p\in \mathbb N_0}\ln\frac{x^{p}}{M^{\tau,\sigma}_p},\quad \;\; x>0.  
\end{align*} 
\end{de}

\begin{rem}
Since
\[
\left(M^{\tau,\sigma}_p\right)^{1/p} = p^{\tau p^{\sigma - 1}},\quad p\in\mathbb N,
\]
is clearly bounded below by a positive constant for any $\tau > 0$ and $\sigma > 1$, it follows that $T_{\tau,\sigma}(x)$ vanishes for sufficiently small $x > 0$ (see \cite{Komatsuultra1}). Thus we may consider $T_{\tau,\sigma}(|x|)$, $x\in \mathbb R$, instead.
\end{rem}

It turns out that for any fixed $\tau>0$, $T_{\tau,\sigma}(|x|)$ is asymptotically equivalent to the function  $\displaystyle g_{\sigma}(x)=  \frac{\ln^{\frac{\sigma}{\sigma-1}} (1+|x|)}{ W^{\frac{1}{\sigma-1}} (\ln(1+|x|))}$ which appears in \eqref{BazicnaFunkcija}. In particular, following Proposition holds (see  \cite{PTT-03, TT-01}).

\begin{prop}
\label{TeoremaAsocirana}
For any $\tau>0$ and $\sigma>1$ there exists $A_{\sigma},B_{\sigma}>0$ and $\widetilde{A}_{\tau,\sigma}, \widetilde{B}_{\tau,\sigma}\in \mathbb R$ such that for $x>0$
\begin{equation*}
\label{ocenaAsocirana}
  B_{\sigma}\tau^{-\frac{1}{\sigma-1}} \frac{\ln^{\frac{\sigma}{\sigma-1} }(1+x) }{W ^{\frac{1}{\sigma-1}}(\ln (1+x))} +\widetilde{B}_{\tau,\sigma}\leq T_{\tau,\sigma}(x)\leq A_{\sigma}\tau^{-\frac{1}{\sigma-1}} \frac{\ln^{\frac{\sigma}{\sigma-1} }(1+x) }{W ^{\frac{1}{\sigma-1}}(\ln (1+x))} +\widetilde{A}_{\tau,\sigma}.
 \end{equation*} 
\end{prop}


\begin{rem}
Let us define 
$$\displaystyle h_{\tau,\sigma}(x):= e^{-T_{\tau,\sigma}(1/x)}=\inf_{p\in \mathbb N_0} M^{\tau,\sigma}_p x^p,\quad x>0.$$ Then Proposition \ref{TeoremaAsocirana} implies that for any $\rho>0$ and $\sigma>1$ there exist constants $A,B,\tau_1,\tau_2>0$ such that
\begin{equation}
\label{ocenah}
A h_{\t_1,\sigma}(x)\leq f_{\rho,\sigma }(x)  \leq B h_{\t_2,\sigma}(x) ,\quad x>0,
\end{equation} where $f_{\rho,\sigma}(x)=e^{-\rho g_{\sigma}(1/x)}$ is given in \eqref{BazicnaFunkcija}. 

Note that the right-hand side of \eqref{ocenah} implies that for any given $\rho>0$ there exist $B_1,\tau>0$ such that
\begin{equation}
\label{OcenafNiz}
\sup_{x>0}\frac{f_{\rho,\sigma}(x)}{x^p}\leq B_1 M^{\tau,\sigma}_p,\quad p\in \mathbb N_0.
\end{equation}
\end{rem}









The following definition introduces the precise decay rate studied in this paper.

\begin{de}
\label{DefinicijaOpadanje}
Let $\sigma>1$ and  $\displaystyle g_{\sigma}(x)=  \frac{\ln^{\frac{\sigma}{\sigma-1}} (1+|x|)}{ W^{\frac{1}{\sigma-1}} (\ln(1+|x|))}$, $x\in \mathbb R$. A continuous function $\phi$ belongs to $\displaystyle {{\mathbf{\Gamma}}}_{\sigma}\left(\mathbb{R}\right)$ if 
\begin{equation}
\label{NejednakostPaley}
(\exists \rho>0)\,(\exists C>0)\quad\quad |\phi(x)| \leq  C\,e^{-\rho g_{\sigma}(x)},\quad  x \in \mathbb R.
\end{equation} 
\end{de}
We can now state the Paley--Wiener theorem, which characterizes the regularity of $\mathcal{E}_{\sigma}$ in terms of the decay rate given by $\mathbf{\Gamma}_{\sigma}$. For further details, we refer to \cite{PTT-03, TT-03}.

\begin{te}
\label{Thm:Peli-Viner}
Let $\sigma>1$ and ${{\mathbf{\Gamma}}}_{\sigma}$ be as in Definition \ref{DefinicijaOpadanje}.

\begin{itemize}
\item[a)] Let $\varphi$ be a $C^{\infty}$ function with compact support. If $\varphi\in \mathcal E_{\sigma}(\mathbb R)$ then $\widehat\varphi  \in{{\mathbf{\Gamma}}}_{\sigma}(\mathbb R)$.

\item[b)] If the function $\varphi$ is such that $\widehat\varphi\in {\mathbf{\Gamma}}_{\sigma}(\mathbb R)$ 
 then $\varphi\in \mathcal E_{\sigma}(\mathbb R)$.
\end{itemize}
\end{te}


We conclude this section with the following technical lemma, which will be used in the sequel.
\begin{lema}
\label{TehnickaLema}
   Let $\sigma>1$ and ${{\mathbf{\Gamma}}}_{\sigma}$ be as in Definition \ref{DefinicijaOpadanje}. Then $\phi(\xi) \in \mathbf{\Gamma}_{\sigma} (\mathbb R)$ if and only if $(1+|\xi|)^a \phi(\xi) \in \mathbf{\Gamma}_{\sigma} (\mathbb R)$ for arbitrary $a \in \mathbb{R}$.
\end{lema}
\begin{proof}
Let $\phi\in \mathbf{\Gamma}_{\sigma} (\mathbb R)$ and $a\in\mathbb{R}$. Then by \eqref{NejednakostPaley}, exist $\rho>0, C>0$ such that
\begin{align*}
    (1+|\xi|)^a |\phi (\xi)|
    &\leq (1+|\xi|)^a  C\exp\left\{-\rho \frac{\ln^{\frac{\sigma}{\sigma-1}} (1+|\xi|)}{ W^{\frac{1}{\sigma-1}} (\ln(1+|\xi|))}\right\}\\
    &= C (1+|\xi|)^a e^{-\rho \,g_{\sigma}(\xi)}  \\
    &\leq C \max\{1, 2^{a}|\xi|^a\} e^{-\frac{\rho}{2}\,g_{\sigma}(\xi)} e^{-\frac{\rho}{2}\,g_{\sigma}(\xi)}\\
    &\leq C_1 e^{-\frac{\rho}{2}\,g_{\sigma}(\xi)},\quad\quad |\xi|\geq 1,
\end{align*}
for $\displaystyle C_1 := C \sup_{\xi \neq 0} \{ \max\{1, 2^{a}|\xi|^a\} e^{-\frac{\rho}{2}\,g_{\sigma}(\xi)} \}.$ Note that $C_1$ is finite, because
$$\sup_{\xi\neq 0} |\xi|^a e^{-\frac{\rho}{2}\,g_{\sigma}(\xi)} = \sup_{t>0}  \frac{f_{\frac{\rho}{2},\sigma}(t)}{t^{a}} <\infty,$$
which follows easily by \eqref{OcenafNiz}. The other implication is trivial.
\end{proof}

\begin{rem}
In \cite{Javier02}, it was noted that the sequence 
\[
M^{\tau,\sigma}_0 =1 \quad \text{and}\quad M^{\tau,\sigma}_p = p^{\tau p^{\sigma}}, \qquad p \in \mathbb{N},
\]
satisfies a strong non-quasianalyticity condition for every choice of $\tau > 0$ and $\sigma > 1$, i.e.,
\begin{align*}
\sum_{q=p}^{\infty} \frac{M_{q}^{\tau,\sigma}}{(q+1)\, M_{q+1}^{\tau,\sigma}}
   \leq C\, \frac{M_{p}^{\tau,\sigma}}{M_{p+1}^{\tau,\sigma}}, 
   \qquad p \in \mathbb{N}_0.  
\end{align*}
Moreover, if $\displaystyle m^{\tau,\sigma}_p = \frac{M^{\tau,\sigma}_{p+1}}{M^{\tau,\sigma}_p}$, $p \in \mathbb{N}_0$, the stronger assertion holds: the gamma index of the sequence $\mathbb M^{\tau,\sigma}=\{M^{\tau,\sigma}_p\}_{p\in \mathbb{N}_0}$, given by
\[
\gamma(\mathbb M^{\tau,\sigma})
   = \sup \left\{ \gamma>0 :\; 
        (\exists C>0)(\forall p\in\mathbb N_0)(\forall q\ge p)\,
        \frac{m^{\tau,\sigma}_p}{(p+1)^{\gamma}}
        \le C\, \frac{m^{\tau,\sigma}_q}{(q+1)^{\gamma}} 
     \right\},
\]
satisfies $\displaystyle \gamma(\mathbb M^{\tau,\sigma}) = \infty$.

This implies that the corresponding function classes $\mathcal{E}_{\tau,\sigma}(\mathbb{R})$, $\tau > 0$, $\sigma > 1$,  admit (ultraholomorphic) extensions to unbounded sectors in $\mathbb{C}$ of the form
\begin{equation}\label{defsektora}
S_{\gamma} := 
\left\{ z \in \mathcal{R} : |\arg(z)| < \frac{\gamma \pi}{2} \right\},
\end{equation}
for arbitrary $\gamma>0$, where $\mathcal{R}$ denotes the Riemann surface of the logarithm.  
\end{rem}

\subsection{An example of the extended Gevrey function}
\label{section:example}
First, we discuss the complex extension to sectors of the functions defined in \eqref{BazicnaFunkcija}. For our purposes, it is sufficient to consider unbounded sectors $S_{\gamma}$ given in \eqref{defsektora} with $0 < \gamma \leq 2$, taking into account only the principal branch of the logarithm.

If $T$ and $S$ are two unbounded sectors, we say that $T$ is a \emph{proper subsector} of $S$ if $\overline{T} \subseteq S \cup \{0\}$. By taking the Riemann surface of the logarithm $\mathcal{R}$ as our initial set, we ensure that the vertex $0$ is not included in ${T}$. For example, note that $S_2 = \mathbb{C} \setminus (-\infty,0]$ and $S_1 = \{ z \in \mathbb{C} \mid \mathrm{Re}(z) > 0 \}$; in this case, $S_1$ is a proper subsector of $S_2$.

We will use the following Lemma from \cite{Javier02}.

\begin{lema}
\label{spanskalema0}
Let $\tau>0$ and $\sigma>1$. Set $\displaystyle a_{\tau , \sigma } = \left(\frac{\sigma - 1}{\tau \sigma}\right)^{\frac{1}{\sigma - 1}}$ and $\displaystyle b_{\tau , \sigma} =  \frac{\sigma - 1}{\tau \sigma} \, e^{\frac{\sigma - 1}{\sigma}}$. 
\begin{itemize}
\item[a)] Function $\displaystyle g_{\tau,\sigma}(x)=\frac{\ln^{\frac{\sigma}{\sigma-1}} (1+x)}{  W^{\frac{1}{\sigma-1}} (b_{\tau,\sigma}\ln(1+x))}$ is positive and strictly increasing on $\mathbb R_{+}$.
\item [b)] Set
 \begin{align*}
 e_{\tau,\sigma}(z)=z \exp \left\{ - a_{\tau , \sigma} W^{-\frac{1}{\sigma-1}}( b_{\tau , \sigma}\ln (1+z)) \ln^{\frac{\sigma}{\sigma-1}}(1+z) \right\},\quad z\in  S_2,  
 \end{align*}
 where $\ln(\cdot)$ denotes a principal branch of ${\rm Ln}[\cdot]$. Then for every proper subsector $T$ in $S_2$  there exists $C_1,C_2,K_1,K_2>0$ such that
\begin{equation}
\label{NejednakostEtausigma}
 C_2 e_{\tau,\sigma}(K_2|z|)\leq|e_{\tau,\sigma}(z)|\leq C_1 e_{\tau,\sigma}(K_1|z|),\quad z\in T.
 \end{equation}
\end{itemize}
\end{lema}

In the next lemma we discuss properties of the functions given in \eqref{BazicnaFunkcija}.


\begin{lema}
    \label{spanskalema}
Let $\sigma>1$. 
\begin{itemize}
\item [a)]  Function $\displaystyle g_{\sigma}(x)=\omega_{\sigma}(\ln (1+x))=\frac{\ln^{\frac{\sigma}{\sigma-1}} (1+x)}{ W^{\frac{1}{\sigma-1}} (\ln(1+x))}$ is positive and strictly increasing on $\mathbb R_{+}$. Moreover, for any $C>0$ we have
\begin{equation}
\label{IzlazakKonstanti}
 g_{\sigma}(C x)\asymp g_{\sigma}(x),\quad x\to \infty.
\end{equation}  
\item[b)] For every $a,b>0$
    \begin{align*}
    \lim_{x\to \infty} e^{ax} e^{-b\omega_{\sigma} (x)} = 0,  \quad
    \omega_{\sigma}(x) \sim x,\,\, x\to 0,\quad
    \lim_{x\to 0^+} \omega_{\sigma}(x) = 0.
    \end{align*}
    In particular,  $\displaystyle \omega_{\sigma}(x) = \frac{x^{\frac{\sigma}{\sigma-1}}}{W^{\frac{1}{\sigma-1}}(x)}$ is continuous on $[0,+\infty)$.

\item[c)]
Function $f_{\rho,\sigma}(x)=\exp\{-\rho g_{\sigma}(1/x)\}$ is strictly increasing and continuous on $[0,\infty)$ with $0 \leq |f_{\rho,\sigma}(x)|<1$. 

For every $\varepsilon>0$
\begin{align*}
    \int_0^{\varepsilon} f_{\rho,\sigma} (x) \, dx < \infty.
\end{align*}

Moreover, a family $\{f_{\rho,\sigma}(x)\}, \rho >0 , \sigma >1$ is decreasing  with respect to $\rho$, and increasing with respect to $\sigma$ for every fixed $x\in \mathbb R_+$.

\item[d)]
For every proper subsector $T$ in $S_2$ there exist $C_1,C_2,\rho',\rho''>0$  such that
\begin{align}
\label{eq.Bounds_e_h_subsectors_real_axis}
C_2 f_{\rho'' , \sigma } \left( |z|\right)
\leq  \left| f_{\rho_{\sigma},\sigma} (z) \right|
\leq C_1 f_{\rho' , \sigma } \left( |z| \right),\quad z\in T,
\end{align} where $\rho_{\sigma}=e^{-1/\sigma}$ and
\begin{align*}
\displaystyle f_{\rho_{\sigma},\sigma} (z)=\exp \left\{ - e^{-\frac{1}{\sigma}} W^{-\frac{1}{\sigma-1}}\left(\ln \left(1+\frac{1}{z}\right)\right) \ln^{\frac{\sigma}{\sigma-1}} \left(1+\frac{1}{z}\right) \right\}, \quad z\in S_2.  
\end{align*}
\end{itemize}

\end{lema}

\begin{proof}
Throughout the proof we set $\displaystyle \t_{\sigma}=\frac{\sigma - 1}{\sigma} \, e^{\frac{\sigma - 1}{\sigma}}$. 


$a)$ Note that $g_{\sigma} = g_{\tau_{\sigma},\sigma}$, where $g_{\tau,\sigma}$ is the function introduced in Lemma~\ref{spanskalema0}. Therefore, $g_{\sigma}$ is positive and strictly increasing. Moreover, \cite[Theorem~3.1]{TT-01} implies that $\displaystyle g_{\sigma}(k) = \omega_{\sigma}(\ln(1+|k|))$, $k>0$, is equivalent to a BMT function (see also Example \ref{exampleBMT}). Hence it satisfies the BMT condition~$(\alpha)$, which implies ~\eqref{IzlazakKonstanti}.




$b)$ By the Lemma \ref{LemaLambert} and $(W2)$ it is obvious that
 \begin{align*}
     &\lim_{x\to +\infty} \exp \left\{ ax - b\frac{ x^{\frac{\sigma}{\sigma - 1}}}{W^{\frac{1}{\sigma - 1}} (x)} \right\}
     = \lim_{x\to +\infty} \exp \left\{ x\left( a - b\left( \frac{x}{W(x)} \right)^{\frac{1}{\sigma-1}} \right) \right\}\\
     &=\lim_{x\to +\infty} \exp \left\{ x\left( a - b\,e^{\frac{1}{\sigma-1}W(x)} \right) \right\}=0.
     \end{align*}
  Moreover, by \eqref{LambertU0}
  \begin{align*}
      \lim_{x\to 0} \frac{\omega_{\sigma} (x)}{x}
      = \lim_{x\to 0} \frac{x^{\frac{\sigma}{\sigma-1}-1}}{W^{\frac{1}{\sigma-1}} (x)}
      = \lim_{x\to 0} \left(\frac{x}{W(x)}\right)^{\frac{1}{\sigma-1}}
      =1.
  \end{align*}

$c)$ The fact that $f_{\rho,\sigma}$ is strictly increasing and continuous on $[0,+\infty)$, follows from $a)$ and $b)$. In particular, for $\varepsilon>0$ we have
\begin{align*}
    \int_0^{\varepsilon} e^{-\rho \omega_{\sigma} \left( \ln \left( 1+\frac{1}{x} \right) \right)} \, dx
    = \int_{\frac{1}{\varepsilon}}^{+\infty} e^{-\rho \omega_{\sigma} \left( \ln \left( 1+s \right) \right)} \frac{1}{s^2} \, ds
    \leq \int_{\frac{1}{\varepsilon}}^{+\infty} \frac{ds}{s^2}
    <+\infty.
\end{align*}
The fact that $f_{\rho , \sigma}$ is a decreasing function with respect to $\rho$ follows directly from the definition of $f_{\rho , \sigma}$ in \eqref{BazicnaFunkcija}. For a fixed $x\in\mathbb R$, we note
\begin{align*}
    F_{x} (\sigma) := g_{\sigma} (x)>0, \quad \sigma>1.
\end{align*}
Then we have
\begin{align*}
   \frac{d}{d\sigma} F_{x} (\sigma)
    = \frac{F_{x} (\sigma)}{(\sigma-1)^2} \ln \frac{W\left(\ln\left( 1+|x| \right)\right)}{\ln \left( 1+|x| \right)}
    <0,
\end{align*}
because $x\geq W(x)$ for all $x>0$. This implies that $f_{\rho , \sigma}$ is an increasing function in terms of $\sigma$.
  
$d)$ Note that $e_{\t_{\sigma},\sigma}(z)=\displaystyle z f_{\rho_{\sigma},\sigma}(1/z)$ and hence \eqref{NejednakostEtausigma} implies that
 \begin{align}
\label{ocena}
    C_2 f_{{\rho_{\sigma}} , \sigma} \left( K_2/|z| \right)
    \leq |f_{{\rho_{\sigma} } , \sigma} \left( 1/z \right)|
    \leq C_1  f_{{\rho_{\sigma} } , \sigma} \left( K_1/|z| \right)
    , \quad z\in T.
\end{align} 
Since  $\arg\frac{1}{z}=-\arg z ,$ $ z\neq 0$, note that \eqref{ocena} remains valid if we put $\omega=1/z\in T$. Since $\displaystyle f_{{\rho_{\sigma} } , \sigma} \left( K |\omega| \right)=e^{-\rho_{\sigma}g_{\sigma}(1/(K|\omega|))}$, $K>0$, we can apply \eqref{IzlazakKonstanti} and \eqref{eq.Bounds_e_h_subsectors_real_axis} follows.
\end{proof}

In the next section we will construct a wavelet using a particular function from $\E_{\sigma}(\mathbb R)$. For $\sigma>1$  let us consider 
\begin{equation}
\label{PrimerExtended}
\displaystyle f_{\rho_{\sigma},\sigma} (x)=\exp \left\{ - \rho_{\sigma
} \,W^{-\frac{1}{\sigma-1}} \left(\ln \left(1+\frac{1}{|x|}\right)\right) \ln^{\frac{\sigma}{\sigma-1}} \left(1+\frac{1}{|x|}\right) \right\},\quad x\in \mathbb R,
\end{equation} where $\rho_{\sigma}=e^{-1/\sigma}$. First we prove that $f_{\rho_{\sigma},\sigma}$ is smooth but not analytic on $\mathbb R$. 

The following Lemma holds.

\begin{lema}
\label{LemmaFlat}
For $\sigma>1$ and $\rho_{\sigma}=e^{-1/\sigma}$, let $f_{\rho_{\sigma},\sigma}$ be as in  \eqref{PrimerExtended}. Then it holds $\lim_{x\to 0} f_{\rho_{\sigma},\sigma} ^{(j)} \left( x \right) = 0$ for all $j\in\mathbb N_0$.
\end{lema}

Proof of the Lemma \ref{LemmaFlat} is rather technical and therefore given in the Appendix. In the following proposition we capture the regularity of $f_{\rho_{\sigma},\sigma}$ by means of the classes $\mathcal{E}_{\sigma}(\mathbb{R})$.



\begin{te}
\label{PrimerExtGevrey}
Let $\sigma>1$, $\rho_{\sigma}=e^{-1/\sigma}$, and let $f_{\rho_{\sigma},\sigma}$ be given by  \eqref{PrimerExtended}. Then $f_{\rho_{\sigma},\sigma} \in\mathcal E_{\sigma}(\mathbb R)$.
\end{te}
\begin{proof}

Since $f_{\rho_{\sigma},\sigma}$ is even on $\mathbb R$ and Lemma \ref{LemmaFlat} holds, we estimate the derivatives of $f_{\rho_{\sigma},\sigma}$ on $\mathbb R_+$. Choose subsector $T$ in $S_2$ (for instance $S_1$) such that for all $x>0$ it holds $D=\{z\in \mathbb C\,\,|\,\,|z-x|\leq \frac{x}{2}\}\subset T$.

Note that $|z-x|=\frac{x}{2}$ implies $\frac{x}{2}\leq |z|\leq\frac{3x}{2}$. Using Cauchy's integral formula and (\ref{eq.Bounds_e_h_subsectors_real_axis}) we obtain
\begin{align}
\label{NejednakostLema1}
    \left| \frac{d^j}{dx^j} f_{\rho_{\sigma},\sigma} (x) \right|
    &= \left| \frac{j!}{2\pi i} \int_{|z-x|
    =\frac{x}{2}} \frac{f_{\rho_{\sigma},\sigma}(z)}{(z-x)^{j+1}} \, dz \right|
    \leq \frac{j!}{2\pi } \int_{|z-x|=\frac{x}{2}} \frac{|f_{\rho_{\sigma},\sigma}(z)|}{|z-x|^{j+1}} \, |dz|\nonumber\\
    &\leq \frac{j!2^j}{\pi x^{j+1}}  \int_{|z-x|=\frac{x}{2}} C_1 \, f_{\rho',\sigma} \left( |z| \right) \, |dz|\nonumber\\
    &\leq \frac{C_1 j!2^j}{\pi x^{j+1}}  \int_{|z-x|
    =\frac{x}{2}} f_{\rho',\sigma} \left( \frac{3x}{2} \right) \, |dz|\nonumber\\
    &= \frac{C_1 j!2^j}{x^j} \, f_{\rho',\sigma} \left( \frac{3x}{2} \right), \quad x>0,\,\, j\in \mathbb N_0,
   \end{align}
for suitable constants  $C_1,\rho>0$, where we also used that $f_{\rho,\sigma}$ are increasing.

If the function $g_{\sigma}$ is given in \eqref{BazicnaFunkcija}, then \eqref{IzlazakKonstanti} implies

\begin{equation}
\label{NejednakostLema2}
 f_{\rho',\sigma} \left( \frac{3x}{2} \right)
 =\exp\left\{-\rho' g_{\sigma}\Big(\frac{2}{3x}\Big)\right\} 
 \leq L \exp\left\{-\rho'' g_{\sigma}\Big(\frac{1}{x}\Big)\right\}
 =L f_{\rho'',\sigma} \left( x \right),
 \end{equation} for $x>0$ and suitable $L,\rho''>0$.

Take $\mathbb K\subset\subset \mathbb R_+$. Using \eqref{NejednakostLema1}, \eqref{NejednakostLema2}, \eqref{OcenafNiz}, and simple inequality $j!\leq C^{j^\sigma}$, $C>1$, $j\in\mathbb N_0$, for suitable $C_2,\tau>0$ we obtain
$$\left| \frac{d^j}{dx^j} f_{\rho_{\sigma},\sigma} (x) \right|\leq C_1 (2C)^{j^{\sigma}} \frac{f_{\rho'',\sigma}\left( x \right)}{x^j}  \leq C^{j^{\sigma}+1}_2 M^{\tau,\sigma}_j,\quad x\in \mathbb K,\;j\in\mathbb N_0,$$ 
and this proves the assertion.
\end{proof}

We conclude this section by proving that the choice of $\rho_{\sigma}=e^{-1/\sigma}$ in \eqref{PrimerExtended} is not essential.

\begin{prop}
\label{proizvoljnorho}
 Let $\sigma>1$. Then for any $\rho>0$ 
 $$ f_{\rho,\sigma}\in \E_{\sigma}(\mathbb R)\backslash \bigcup_{1<\sigma'<\sigma}\E_{\sigma'}(\mathbb R),$$ 
 where $f_{\rho,\sigma}$ are given in \eqref{BazicnaFunkcija}.
\end{prop}

\begin{proof}
Take an arbitrary $\rho>0$. It is clear that Lemma~\ref{LemmaFlat} also holds for
\[
f_{\rho,\sigma} = (f_{\rho_{\sigma},\sigma})^{\rho/\rho_{\sigma}}.
\]
Moreover, since the function $z^{\alpha}$, $\alpha>0$, is analytic on $S_2=\mathbb{C}\setminus(-\infty,0]$,  Proposition~\ref{OsobineKlasa} $e)$ implies that $\displaystyle f_{\rho,\sigma}\in \mathcal{E}_{\sigma}(\mathbb{R})$, and therefore Proposition~\ref{PrimerExtGevrey} holds as well.

Let us now show that
\[
f_{\rho,\sigma}\notin \bigcup_{1<\sigma'<\sigma}\mathcal{E}_{\sigma'}.
\]
Assume the opposite, that $f_{\rho,\sigma}\in \mathcal{E}_{\sigma'}$ for some $\sigma'<\sigma$. Then, for a small $t>0$, using \eqref{M4} we obtain
\[
f_{\rho,\sigma}(x)
   = \int_{0}^{x} \frac{(x-s)^{p-1}}{(p-1)!} f^{(p)}_{\rho,\sigma}(s)\, ds
   \leq \frac{x^{p}}{p!}\, C^{p^{\sigma'}+1} M^{\tau_0,\sigma'}_{p}
   \leq C'\, x^{p}\, M^{\tau,\sigma}_{p},\qquad p\in\mathbb{N},
\]
for arbitrary $\tau>0$ and some constant $C'>0$ (depending on $\tau$).
This contradicts the left-hand side of \eqref{ocenah}.
\end{proof}

\begin{rem}
Note that Proposition~\ref{OsobineKlasa} $a)$ and Proposition~\ref{proizvoljnorho} imply that, for arbitrary $\rho>0$ and $\sigma>1$, holds $\dss f_{\rho,\sigma} \notin \bigcup_{t>1} \mathcal{G}_{t}$, where $\mathcal{G}_{t}$ denotes the Gevrey classes.
\end{rem}





\section{Wavelet construction}
\label{section:waveletconstruction}
In this section, we denote by $m_{0}(\xi)$ a low-pass filter of an orthonormal wavelet satisfying the following minimal requirements:
\begin{align}
\label{CozM0}
\begin{split}
    &\textit{i)} \
    m_0 \text{ is a continuous, even, $2\pi-$ periodic function on } \mathbb R, m_0 (0)=1,
    \\
    &\textit{ii)} 
    \inf_{\xi\in[-\pi/3,\pi/3 ]} |m_0(\xi)|\not=0,
    \\
    &\textit{iii)} \
    |m_0(\xi)|^2+|m_0(\xi+\pi)|^2=1.
\end{split}
\end{align}







 Recall that ${\rm coz}\, m_0$  denotes the cozero set (complement of a zero set) of $m_0$. Let us start with the following results, see \cite[Lemmas 2.3 and 3.1]{Fukuda} and \cite[Theorem 1]{Fukuda2}.

\begin{lema}
\label{LemmaFukuda}
Let $r>0$, $n\in \mathbb N$, and let $m_0$ be a low-pass filter that satisfies conditions $i)$, $ii)$ and $iii)$ in \eqref{CozM0} such that $\operatorname{coz} m_0 \cap[0, \pi] \subset\left[0,\left(\frac{2}{3}+r\right) \pi\right)$. If $\varphi$ is a scaling function associated with MRA, then for  $r \leq \frac{2}{15}$ it holds
$$
\operatorname{coz} \hat{\varphi} \cap\left[2^n \pi, 2^{n+1} \pi\right] \subset\left(\frac{2^{n+2}}{3} \pi-\left(3+(-1)^n\right) r \pi, \frac{2^{n+2}}{3} \pi+\left(3-(-1)^n\right) r \pi\right) .
$$
In particular, if $r = \frac{2}{15}$ then $\operatorname{coz} m_0 \cap[0, \pi]\subset \left[0, \frac{4}{5} \pi\right)$ and
\begin{equation}
\label{NejednakostFukuda1}
\displaystyle \left|\xi-\frac{2^{n+2}}{3} \pi\right|<\frac{8}{15} \pi, \quad \xi \in \operatorname{coz} \hat{\varphi}\, \cap \left[2^n \pi, 2^{n+1} \pi\right].
\end{equation} In addition, for $\xi \in \operatorname{coz} \hat{\varphi}\, \cap \left[2^n \pi, 2^{n+1} \pi\right]$ following inequality holds
\begin{multline}
\label{NejednakostFukuda2}
\inf _{\xi \in A_n}\left|m_0(\xi)\right|\left|\prod_{j=1}^{n+1} m_0\left(\frac{2}{3} \pi-(-1)^{n-j} \frac{\xi-\frac{2^{n+2}}{3} \pi}{2^j}\right)\right| \leq|\hat{\varphi}(\xi)|\\
\leq\left|\prod_{j=1}^{n+1} m_0\left(\frac{2}{3} \pi-(-1)^{n-j} \frac{\xi-\frac{2^{n+2}}{3} \pi}{2^j}\right)\right|,
\end{multline}
where $\displaystyle A_n=\left(\frac{\pi}{3}-\frac{3+(-1)^n}{2^{n+2}} \pi, \frac{\pi}{3}+\frac{3-(-1)^n}{2^{n+2}} \pi\right)$, $n\in \mathbb N$.

If $\displaystyle\frac{2\pi}{3}\not\in {\rm int}({\rm supp}\,\, m_0)$ then $\varphi$ is band-limited.

\end{lema}

\begin{rem}
\label{Jednako1}
Note that the condition
\[
\frac{2\pi}{3} \notin \operatorname{int}(\operatorname{supp} m_{0})
= \operatorname{int}\big(\overline{\operatorname{coz} m_{0}}\big),
\]
implies that \(m_{0}\) vanishes in a neighborhood of \(\frac{2\pi}{3}\).
In our construction, we prescribe the local decay of \(m_{0}\) near \(\frac{2\pi}{3}\), so that
\[
\frac{2\pi}{3} \notin \operatorname{coz} m_{0},
\qquad \text{but} \qquad
\frac{2\pi}{3} \in \operatorname{int}(\operatorname{supp} m_{0}).
\]


If $\operatorname{coz} m_{0} \cap [0, \pi] \subset \left[0, \frac{4}{5} \pi\right)$, note that property~$iii)$ in \eqref{CozM0} implies
\[
|m_{0}(\xi)| = 1 \quad \text{for } \xi \in \left[0, \frac{\pi}{5}\right].
\]
Moreover, $|m_{0}(\xi)| \leq 1$ for all $\xi \in \mathbb{R}$, and since $\displaystyle m_{0}\big(-\frac{2\pi}{3}\big) = 0$, we also have $\displaystyle \Big|m_{0}\left(\frac{\pi}{3}\right)\Big| = 1$. This yields to non-band-limited scaling function $\varphi$ given in \eqref{OdnosM0iPhi}.
\end{rem}

In the following theorem, we control the local behavior of \( m_0(\xi) \) using the functions $f_{\rho,\sigma}$ defined in \eqref{BazicnaFunkcija}, to obtain a global decay estimates for $\hat\varphi$. Moreover, we consider $\displaystyle \sigma_{\eta}=\frac{\sigma+\eta (\sigma -1)}{1+\eta (\sigma-1)}$, $\sigma>1$, $\eta>1$. Note that $\sigma_{\eta}\in (1,\sigma_{0})$ where $\sigma_0 = \min \{2,\sigma\}$.

\begin{te}
\label{TeoremaLowPass}
   Let $m_0$ be a low-pass filter that satisfies conditions $i)$, $ii)$ and $iii)$ in \eqref{CozM0} such that $\operatorname{coz} \, m_0 \cap [0,\pi] \subset [0,\frac{4\pi}{5})$ and $\varphi$ be the associated scaling function. Moreover, let $\rho>0$, $\s>1$, $f_{\rho,\sigma}$ be as in \eqref{BazicnaFunkcija} and ${{\mathbf{\Gamma}}}_{\sigma}$ as in Definition \ref{DefinicijaOpadanje}.

   Then the following is true: if $m_0$ satisfies
   
    \begin{align}
    \label{uslovte}
    \begin{split}
        C_1 f_{\rho_1 , \sigma} \left(  \xi - \frac{2\pi}{3} \right)
        \leq |m_0 (\xi)|
        \leq C_0 f_{\rho_0 , \sigma} \left( \xi - \frac{2\pi}{3} \right),
        \quad \xi\in \left[ \frac{2\pi}{3}-\varepsilon , \frac{2\pi}{3}+\varepsilon \right],
    \end{split}
    \end{align} for some constants $\varepsilon,C_0, C_1,\rho_1,\rho_0>0$
    then 
    $$
    \widehat{\varphi}\in {{\mathbf{\Gamma}}}_{\sigma}\backslash \bigcup_{1<\sigma'<\sigma_{\eta} } {{\mathbf{\Gamma}}}_{\sigma'},$$ for each $\eta>1$. Consequently, $\displaystyle \varphi\in {{\mathcal{E}}}_{\sigma}\backslash \bigcup_{1<\sigma'<\sigma_{\eta} } {{\mathcal{E}}}_{\sigma'} $.
\end{te}
\begin{proof}
      Without loss of generality, we may choose $\varepsilon<\frac{8\pi}{15}$ so that \eqref{uslovte} holds.  Moreover, let $n_0\in\mathbb{N}$ be such that $\displaystyle \frac{8\pi}{15 \cdot 2^{n_0}}< \varepsilon$. Then for all $\xi\in \operatorname{coz} \, \widehat{\varphi} \cap [2^n \pi , 2^{n+1}\pi]$ and $n\geq n_0$, the right-hand side of \eqref{NejednakostFukuda2} implies
    \begin{align}
    \label{ocenaFurijePhi}
        |\widehat{\varphi} (\xi)|
        &\leq \left| \prod_{j=1}^{n+1} m_0 \left( \frac{2\pi}{3} - (-1)^{n-j} \frac{\xi - \frac{2^{n+2}}{3}\pi}{2^j} \right) \right|\nonumber \\
        &\leq \left| \prod_{j=n_0}^{n+1} m_0 \left( \frac{2\pi}{3} - (-1)^{n-j} \frac{\xi - \frac{2^{n+2}}{3}\pi}{2^j} \right) \right|,
    \end{align}
     where in the second inequality we used the fact that $|m_0 (\xi)|\leq 1 ,$ $ \xi\in\mathbb{R}$, see Remark \ref{Jednako1}. 
     
     Further note that for all $n_0\leq j\leq n+1$ and $\xi\in \operatorname{coz} \, \widehat{\varphi} \cap [2^n \pi , 2^{n+1}\pi]$, \eqref{NejednakostFukuda1} implies
    \begin{equation}
    \label{NejednakostTacka}
    \left|  (-1)^{n-j+1} \frac{\xi - \frac{2^{n+2}}{3}\pi}{2^j} \right| 
    = \frac{|\xi-\frac{2^{n+2}}{3}\pi |}{2^j}
    \leq \frac{\frac{8\pi}{15\cdot 2^{n_0}}}{2^{j-n_0} }
    \leq\frac{\varepsilon}{2^{j-n_0}},\quad n_0\leq j\leq n+1.
    \end{equation}  
    Now from (\ref{uslovte}) we have
    \begin{align}
    \label{OcenaFurijePhi2}
        \left|m_0 \left( \frac{2\pi}{3} - (-1)^{n-j} \frac{\xi - \frac{2^{n+2}}{3}\pi}{2^j} \right) \right|
        \leq C_0 f_{\rho_0 , \sigma} \left( (-1)^{n-j+1} \frac{\xi - \frac{2^{n+2}}{3}\pi}{2^j} \right)\nonumber\\
        =C_0 f_{\rho_0 , \sigma} \left(  \frac{|\xi-\frac{2^{n+2}}{3}\pi |}{2^j}\right)\leq C_0 f_{\rho_0 , \sigma} \left( \frac{\varepsilon}{2^{j-n_0}} \right),\quad n_0\leq j\leq n+1,
    \end{align} 
    for $\xi\in \operatorname{coz} \, \widehat{\varphi} \cap [2^n \pi , 2^{n+1}\pi]$, where for the last inequality we used \eqref{NejednakostTacka} and the fact that $f_{\rho,\sigma}$ is even, and increasing on $\mathbb R_+$ (see Lemma \ref{spanskalema} $a)$). 
   
    Moreover, since $\displaystyle f_{\rho_0 , \sigma} \left( \frac{\varepsilon}{2^{j-n_0}} \right)=\exp\Big\{-\rho_0 g_{\sigma}\left( \frac{2^{j-n_0}}{\varepsilon} \right)\Big\}$,  using again Lemma \ref{spanskalema} $a)$, for  $\xi\in \operatorname{coz} \, \widehat{\varphi} \cap [2^n \pi , 2^{n+1}\pi]$ we obtain
    \begin{align}
    \label{ocenaFurijePhi3}
        \sum_{j=1}^{n+1} g_{\sigma} \left( \frac{2^{j-n_0}}{\varepsilon} \right) 
        &\geq g_{\sigma} \left( \frac{2^{n+1-n_0}}{\varepsilon} \right)\geq g_{\sigma}\left( \frac{1}{2^{n_0}\varepsilon \pi} \xi\right)\geq C_1 g_{\sigma}  \left( \xi \right),
    \end{align} where the last inequality follows from \eqref{IzlazakKonstanti}. 
    
    Finally using \eqref{ocenaFurijePhi}, \eqref{OcenaFurijePhi2}, \eqref{ocenaFurijePhi3}, and noting that  $2^n \pi \leq |\xi| \leq 2^{n+1} \pi $ is equivalent to  $n + \log_2 \pi \leq \log_2 |\xi| \leq n+1 + \log_2 \pi$, we obtain
    \begin{align*}
        |\widehat{\varphi} (\xi)|
        &\leq  \prod_{j=n_0}^{n+1} C_0 f_{\rho_0 , \sigma} \left( \frac{\varepsilon}{2^{j-n_0}} \right)
        = C \prod_{j=1}^{n+1} C_0 f_{\rho_0 , \sigma} \left( \frac{\varepsilon}{2^{j-n_0}} \right)\\
        &= C C_0^{n+1} \exp \left\{ -\rho_0 \sum_{j=1}^{n+1}  g_{\sigma} \left( \frac{2^{j-n_0}}{\varepsilon} \right) \right\}\\
        &\leq C C_0^{n+1} \exp \left\{ - C_1\rho_0  g_{\sigma} \left( \xi \right) \right\} \\
        &\leq C C_0^{\log_2 |\xi| + 1 - \log_2 \pi }   \exp \left\{ - \rho'_0  g_{\sigma} \left( \xi \right) \right\} \\
        &\leq C' |\xi|^{\nu_0} \exp \left\{ - \rho'_0  g_{\sigma} \left( \xi \right) \right\},
    \end{align*} 
    where $\rho'_0=C_1 \rho_0$, $\displaystyle C^{-1}=\prod_{j=1}^{n_0-1} C_0 f_{\rho_0 , \sigma} \left( \frac{\varepsilon}{2^{j-n_0}} \right)$, $\displaystyle C'= C C_0^{1 - \log_2 \pi }$ and $\displaystyle \nu_0 = \log_2 C_0$. Now Lemma \ref{TehnickaLema} implies that $\hat\varphi\in {\mathbf\Gamma}_{\sigma}$. 

     It remains to show that for any $\eta>1$, $\displaystyle \widehat{\varphi}\notin \bigcup_{1<\sigma'<\sigma_{\eta} } {\mathbf\Gamma}_{\sigma'}$, where $\sigma_{\eta}=\frac{\sigma+\eta (\sigma -1)}{1+\eta (\sigma-1)}$.

It is sufficient to evaluate $\hat\varphi$ at the points 
    \begin{equation}
    \label{TackeXiN}
    \xi_n := \frac{2^{n+2}}{3}\pi - (-1)^n \varepsilon \in \operatorname{coz} \, \widehat{\varphi} \cap [ 2^n \pi , 2^{n+1} \pi],\quad n\in \mathbb N,\quad \varepsilon<\frac{8\pi}{15}.
    \end{equation} In particular, we will prove that for arbitrary $1<\sigma'<\sigma_{\eta}$ and $C,\rho>0$ there exists $n'_0\in \mathbb N$ such that
    \begin{equation}
    \label{TrebaPokazati}
    |\widehat{\varphi} (\xi_n)|
    \geq C\, \exp\{-\rho\, g_{\sigma'}(\xi_n)\},\quad n\geq n'_0,
\end{equation} where $\xi_n$ are given in \eqref{TackeXiN}.

 Since  $m_0 \left( \frac{\pi}{3} \right) = 1$ (see Remark \ref{Jednako1}), there exist $n_1\in \mathbb{N}$ such $ \inf_{\xi\in A_n} |m_0 (\xi)| >0$ for all $n\geq n_1$, where $A_n$ is given in Lemma \ref{LemmaFukuda}.

Set $\tilde{C}=\inf_{\xi\in A_n} |m_0 (\xi)|$, and note that it depends only on $n_1$. From Lemma \ref{LemmaFukuda}, (\ref{uslovte}) and $n+\log_2 \pi \leq |\xi| \leq n+1+\log_2 \pi$ for $n\geq n_1$ we obtain
    \begin{align}
    \label{DonjaOcena1}
        |\widehat{\varphi} (\xi_n)|
        &\geq \inf_{\xi\in A_n} |m_0 (\xi)| \left| \prod_{j=1}^{n+1} m_0 \left( \frac{2\pi}{3} - (-1)^{n-j} \frac{\xi_n - \frac{2^{n+2}}{3} \pi}{2^j} \right) \right|\nonumber\\
        &= \tilde{C} \left| \prod_{j=1}^{n+1} m_0 \left( \frac{2\pi}{3} - \frac{(-1)^{n-j}}{2^j} \left(\frac{2^{n+2}}{3}\pi - (-1)^n \varepsilon - \frac{2^{n+2}}{3} \pi \right)\right) \right|\nonumber\\
        &= \tilde{C} \left| \prod_{j=1}^{n+1} m_0 \left( \frac{2\pi}{3} + (-1)^j \frac{\varepsilon}{2^j} \right) \right|\nonumber\\
        &\geq \tilde{C} C^{n+1}_1 \prod_{j=1}^{n+1}  f_{\rho_1 , \sigma} \left(  (-1)^j \frac{\varepsilon}{2^j}  \right)\nonumber\\
        &\geq \tilde{C_1} |\xi_n|^{\nu_1} \exp \left\{ -\rho_1 \sum_{j=1}^{n+1} g_{\sigma}\left(\frac{2^j}{\varepsilon} \right) \right\}
    \end{align} 
    for $\tilde C_1 = \tilde{C} C_1 ^{-\log_2 \pi}$ and $\nu_1 = \log_2 C_1$.
    
Take arbitrary $\eta>1$ set $\sigma_{\eta}=\frac{\sigma+\eta(\sigma-1)}{1+\eta (\sigma-1)}\in (1,\sigma_0)$. Moreover, if we denote $\displaystyle a_j=\ln \left(1+\frac{2^j}{\varepsilon}\right)$, then $g_{\sigma}\left(\frac{2^j}{\varepsilon} \right)=\omega_{\sigma}(a_j)$, where $\omega_{\sigma}$ is given in \eqref{BazicnaFunkcija}. Hence
using  property $(W2)$ of the Lambert $W$ function, Lemma \ref{spanskalema} $a)$, and that $2^n \pi\leq \xi_n\leq 2^{n+1}\pi$ we obtain
\begin{align}
\label{DonjaOcena2}
&\sum_{j=1}^{n+1} g_{\sigma}\left(\frac{2^j}{\varepsilon} \right) =\sum_{j=1}^{n+1} \omega_{\sigma} \left( a_j \right)=\sum_{j=1}^{n+1} a_j \exp\left\{{\frac{1}{\sigma-1}W(a_j)}\right\}\nonumber \\
&\leq a_{n+1} \exp\left\{\frac{1 +\eta(\sigma-1)}{\sigma-1}W(a_{n+1})\right\} \cdot \sum_{j=1}^{n+1} \exp\{-\eta W(a_j)\}\nonumber \\
&\leq a_{n+1}\exp\left\{\frac{1}{\sigma_{\eta}-1}W(a_{n+1})\right\}\sum_{j=1}^{\infty}\frac{W^{\eta}(a_j)}{a^{\eta}_j}\nonumber\\
&= C_{\eta}\, g_{\sigma_{\eta}} \left( \frac{2^{n+1}}{\varepsilon} \right)\leq C_{\eta}\, g_{\sigma_{\eta}} \left( \frac{2\xi_n}{\pi \varepsilon} \right)\leq C' C_{\eta}\, g_{\sigma_{\eta}} \left( \xi_n \right),
\end{align} 
for suitable $C'>0$ where $\displaystyle C_{\eta}:=\sum_{j=1}^{\infty}\frac{W^{\eta}(a_j)}{a^{\eta}_j}$. Let us show that $C_{\eta}<\infty$.

 We need to prove that the series $ \sum_{j=1}^{\infty}\frac{W^{\eta}(a_j)}{a^{\eta}_j}$ is convergent. Note that 
 $$a_j\sim  j\ln 2-\ln\varepsilon\sim j\ln2,\quad j\to \infty,$$
 and therefore by \eqref{PosledicaLambert1.5} 
 $$\frac{W^{\eta}(a_j)}{a^{\eta}_j} \sim \frac{\ln^{\eta} (j \ln 2)}{j^{\eta} \ln^{\eta} 2 }\sim C \frac{\ln^{\eta}j}{j^{\eta}},\quad j\to \infty,$$
 where $C=1/\ln^{\eta} 2$. Clearly series $\sum\frac{\ln^{\eta}j}{j^{\eta}}$ is convergent for $\eta>1$ and hence $C_{\eta}<\infty$.
 

Now using \eqref{DonjaOcena1} and \eqref{DonjaOcena2} we obtain
\begin{equation}
\label{PrethodnaNejedn}
|\widehat{\varphi} (\xi_n)|\geq\tilde{C_1} |\xi_n|^{\nu_1} \exp\{-\rho'_1\, g_{\sigma_{\eta}}(\xi_n)\} \geq \tilde{C}_2 \exp\{-\rho'_2\, g_{\sigma_{\eta}}(\xi_n)\}, \quad n\geq n_1,
\end{equation} for suitable $\tilde{C}_2 , \rho'_2>0$, where we used Lemma \ref{TehnickaLema} and $\xi_n$ are given in \eqref{TackeXiN}. 


Take arbitrary $\rho>0$ and $1<\sigma'<\sigma_{\eta}$. Then we can write 
$$\frac{1}{1-\sigma'} = \frac{1}{1-\s_{\eta}}+c_{\sigma'},$$ for suitable $c_{\sigma'}>0$. Moreover, since 
$$\lim_{\xi\to\infty} e^{-c_{\sigma'} W(\ln (1+|\xi|))} = 0,$$ we can choose $|\xi_0|$ large enough to obtain $e^{-c_{\sigma'} W(\ln (1+|\xi|))}< \rho$ when $|\xi|>|\xi_0|$. Therefore, using the property $(W2)$ we have
\begin{align}
\label{PrethodnaNejedn3}
    g_{\s_{\eta}}(\xi)
    &= \ln (1+|\xi|) \exp\left\{\frac{1}{\s_{\eta}-1} W(\ln (1+|\xi|))\right\}\nonumber\\
    &= \ln (1+|\xi|) \exp\left\{\frac{1}{\sigma'-1} W(\ln (1+|\xi|))\right\}\nonumber\\
    & \, \quad \exp\left\{-c_{\sigma'} W(\ln (1+|\xi|))\right\}\nonumber\\
    &\leq \rho \ln (1+|\xi|) \exp\left\{\frac{1}{\sigma'-1} W(\ln (1+|\xi|))\right\},\nonumber\\
    &=\rho g_{\sigma'}(\xi),\quad |\xi|>|\xi_0|.
    \end{align}

Choose $n_2\in \mathbb N$ such that $|\xi_n| >|\xi_0|$ for $n\geq n_2$, where $\xi_n$ are given in \eqref{TackeXiN}. Now \eqref{TrebaPokazati} follows from \eqref{PrethodnaNejedn}, and \eqref{PrethodnaNejedn3} for $n'_0=\max\{n_1,n_2\}$.

The fact that $\displaystyle \varphi\in {{\mathcal{E}}}_{\sigma}\backslash \bigcup_{1<\sigma'<\sigma_{\eta} } {{\mathcal{E}}}_{\sigma'} $ now follows from Theorem \ref{Thm:Peli-Viner}.
\end{proof}

\begin{rem}
    Although Theorem \ref{TeoremaLowPass} holds for all $\sigma>1$, the result is optimal when $1<\sigma\leq 2$. This is due to the restriction $\eta>1$ in $\sigma_{\eta}=\frac{\sigma+\eta (\sigma -1)}{1+\eta (\sigma-1)}$, that we needed because constant $C_{\eta}$ that appear in \eqref{DonjaOcena2} is not finite for $0<\eta\leq 1$. 

\end{rem}

Now we are ready to prove our main result.

\begin{te}
\label{MainResult}
For given  $\sigma>1$ there exists an orthonormal wavelet $\psi$ such that 
$$\displaystyle \psi\in \E_{\sigma}(\mathbb R)\backslash \bigcup_{1<\sigma'<\s_{\eta}}\E_{\sigma'}(\mathbb R) \quad{\rm and}\quad \displaystyle \hat\psi\in \E_{\sigma}(\mathbb R)\backslash \bigcup_{1<\sigma'<\sigma}\E_{\sigma'}(\mathbb R),$$ for each $\eta>1$ where $\displaystyle \sigma_{\eta}=\frac{\sigma+\eta (\sigma -1)}{1+\eta (\sigma-1)}$.
\end{te} 

\begin{proof}
It is sufficient to construct a low-pass filter 
$m_0 \in \mathcal{E}_{\sigma} \setminus \bigcup_{1 < \sigma' < \sigma} \mathcal{E}_{\sigma'}$ 
that satisfies the assumptions of Theorem~\ref{TeoremaLowPass}. The decay and regularity of $\hat\psi$ then follow from \eqref{MRAWavelet}.

Let
\begin{align*}
\gamma_{\sigma} (\mu) :=
\begin{cases}
f_{\sigma} (\mu), \quad & \mu>0 \\
0, \quad & \mu \leq 0,
\end{cases}
\end{align*}
where $f_{\sigma}=f_{1,\sigma}$ is given in \eqref{BazicnaFunkcija}. Note Proposition \ref{proizvoljnorho} implies that $\gamma_{\sigma}\in \mathcal{E}_{\sigma} (\mathbb{R})$. 

Let us define 
\begin{align*}
\delta_{\sigma} (\xi) :=
\begin{cases}
\left( \int_0^1 \gamma_{\sigma} (\mu) \gamma_{\sigma} (1-\mu) \, d\mu \right)^{-1} \int_0^{\xi} \gamma_{\sigma} (\mu) \gamma_{\sigma} (1-\mu) \, d\mu,  \quad & \xi>0 \\
0, \quad & \xi \leq 0.
\end{cases}
\end{align*}
It is clear that $\delta_{\sigma} (\xi) = 1$ for all $\xi\geq 1$, and note $\delta_{\sigma}\in \mathcal{E}_{\sigma} (\mathbb{R})$. This follows from the property $\widetilde{(M.2)'}$ of the $M^{\tau,\sigma}_p$, and Proposition \ref{OsobineKlasa} $c)$.  By choosing sufficiently small $\varepsilon>0$ and $0<\rho,\rho'<1$, we have 
\begin{align}
\label{ocenaGamma}
    \gamma_{\sigma} (\mu)
    &\leq f_{\rho , \sigma} (\xi) f_{1-\rho , \sigma} (\mu)\nonumber , \\
    \gamma_{\sigma} (1-\mu)
    &\geq f_{\rho' , \sigma} (\xi) f_{1-\rho' , \sigma} (1-\mu),\quad \mu\leq \xi < \varepsilon.
\end{align} This follows from the fact that $1-\mu \geq 1-\xi \geq \xi$ when $\xi\in[0,\varepsilon)$, and because $f_{\rho,\sigma}$ are increasing (see Lemma \eqref{spanskalema} $a)$). 

Since $\displaystyle \int_{0}^{\varepsilon} f_{\rho,\sigma}(\mu)d\mu<\infty$, for all $\rho>0$ (see Lemma \ref{spanskalema} $c)$), the estimates in \eqref{ocenaGamma} imply that there exist $\rho_1, \rho'_1,C,C'>0$ such that 
\begin{align*}
    C f_{\rho_1 , \sigma} (\xi)
    \leq \delta_{\sigma} (\xi)
    \leq C' f_{\rho'_1 , \sigma} (\xi), \quad \xi \in [0,\varepsilon).
\end{align*} For $\sigma>1$ we define
\begin{align}
\label{theta}
    \theta_{\sigma} (\xi)
    := \left(1-\delta_{\sigma} \left( \frac{5\xi - \pi}{3\pi} \right) \right) \delta_{\sigma} \left( \frac{|\xi - \frac{2\pi}{3}|}{d} \right) , \quad \xi\in \left[ \frac{\pi}{2} , \pi \right), \; d\in \left( 0 , \frac{\pi}{6} \right].
\end{align}
This function can be extended on $[0,2\pi]$ so that $\theta_{\sigma} (\xi) + \theta_{\sigma} (\xi+\pi)=1$, and then extended further on $\mathbb{R}$ so that it is $2\pi$ periodic. 

The desired low-pass filter $m_0$ is then given by
\begin{align}
\label{NasLowPass}
    m_0 (\xi) = \sin \left( \frac{\pi}{2} \theta_{\sigma} (\xi) \right), \quad \xi\in\mathbb{R}.
\end{align} 

Indeed, note that Proposition \ref{OsobineKlasa} $c)$ and  $e)$ applied to $\E_{\sigma}$, and Proposition \ref{proizvoljnorho}  imply that $m_0 \in \mathcal{E}_{\sigma} \setminus \bigcup_{1 < \sigma' < \sigma} \mathcal{E}_{\sigma'}$.


Let us prove that $m_0$ in \eqref{NasLowPass} satisfies the assumptions of Theorem~\ref{TeoremaLowPass}. For all $\xi\in\mathbb{R}$ it holds
\begin{align*}
    m_0 (\xi+\pi) 
    = \sin \left( \frac{\pi}{2} \theta_{\sigma} (\xi+\pi) \right)
    = \sin \left( \frac{\pi}{2} - \frac{\pi}{2} \theta_{\sigma} (\xi) \right)
    = \cos \left( \frac{\pi}{2} \theta_{\sigma} (\xi) \right),
\end{align*}
and therefore $m_0 ^2 (\xi) + m_0 ^2 (\xi+\pi) = 1$, which is condition $iii)$ in \eqref{CozM0}. Conditions $i)$ and $ii)$ follows directly form the construction.

Note that the functions $m_0 (\xi)$ and $\theta_{\sigma} (\xi)$ have the same decay rate near $\xi = \frac{2\pi}{3}$ , because $\displaystyle \lim_{\xi\to\frac{2\pi}{3}} \frac{m_0 (\xi)}{\theta_{\sigma} (\xi)} = \frac{\pi}{2}$. Moreover, $\theta_{\sigma} (\xi)$ satisfies the inequalities \eqref{uslovte}. This follows since the term $\left(1-\delta_{\sigma} \left( \frac{5\xi - \pi}{3\pi} \right)\right)$ in \eqref{theta} is bounded on $\mathbb R$, and it is strictly positive in the neighborhood of $\xi_0=\frac{2\pi}{3}$. Therefore, $m_0(\xi)$ also satisfies \eqref{uslovte}.
\end{proof}

For $\sigma>1$ let us define
\begin{equation}
\label{Nx}
N_x^\sigma:=\left\{f \in L^2\left(\mathbb{R}\right)\,\, \Big| \,\,\exp \left\{\rho\,\log^{\frac{\sigma}{\sigma-1}} |\xi|\right\} |\hat{f}(\xi)|  \leq C \text{ for some } \rho>0\right\}.
\end{equation} In \cite[Theorem~4.4]{Fukuda}, the authors also constructed a wavelet
\(\psi \in N_x^2\) by imposing polynomial estimates on \(m_0\) in a neighborhood
of \(\frac{2\pi}{3}\).
This approach allows them to obtain improved regularity of the wavelet in the
frequency domain.
On the other hand, our wavelet has regularity of the class
\(\mathcal{E}_{\sigma}(\mathbb{R})\) in both the time and frequency domains,
as a consequence of controlling \(m_0\) by the functions given in
\eqref{BazicnaFunkcija}.

To conclude this paper we compare our regularity with the one proposed by $N^{\sigma}_x$. We start with the following Remark.

\begin{rem}

Let $\displaystyle T(k)=\exp \left\{\rho\,\log^{\frac{\sigma}{\sigma-1}} k\right\},\, k>1,\,\sigma>1,$ be the function that describes the decay rate the decay of $\hat f$ in \eqref{Nx}. Then by the calculation done in \cite[Section 4.1]{Javier03}, it follows that $T(k)$ is associated function for $M_p=q^{p^{\sigma}}$, for suitable $q>1$ when $1<\sigma\leq 2$.
\end{rem}



Now we can compare regularities in the following Lemma.

\begin{lema}
\label{PoslednjaLema}
    Let $N_x^\sigma$ be defined in \eqref{Nx}, and
    \begin{align*}
         \mathcal{E}_{\sigma} ^{\Gamma} 
         = \{ f\in L^2 (\mathbb R) \mid \hat{f}\in {{\mathbf{\Gamma}}}_{\sigma} (\mathbb R) \},
    \end{align*}
    for every $\sigma>1$, where ${\mathbf{\Gamma}}_{\sigma} (\mathbb R)$ is given in the Definition \ref{DefinicijaOpadanje}. Then
    \begin{align*}
        N_x^\sigma
        \subset
        \mathcal{E}_{\sigma} ^{\Gamma}
        \subset
        \mathcal{E}_{\sigma}.
    \end{align*}
\end{lema}
\begin{proof}
    Let $\sigma>1$ and $|\xi_0|>\frac{1+\sqrt{5}}{2}$ be fixed. Then for $|\xi|>|\xi_0|$ we have
    \begin{align*}
        \log^{\frac{\sigma}{\sigma-1}} |\xi|
        &= W^{\frac{1}{\sigma-1}} (\log |\xi_0|) \frac{1}{W^{\frac{1}{\sigma-1}} (\log |\xi_0|)} \log^{\frac{\sigma}{\sigma-1}} |\xi|\\
        &\geq \frac{W^{\frac{1}{\sigma-1}}(\log |\xi_0|)}{2^{\frac{\sigma}{\sigma-1}}} \frac{\log^{\frac{\sigma}{\sigma-1}} (|\xi|^2)}{W^{\frac{1}{\sigma-1}}(\log (1+|\xi|))}\\
        &\geq \frac{W^{\frac{1}{\sigma-1}}(\log |\xi_0|)}{2^{\frac{\sigma}{\sigma-1}}} \frac{\log^{\frac{\sigma}{\sigma-1}} (1+|\xi|)}{W^{\frac{1}{\sigma-1}}(\log (1+|\xi|))},
    \end{align*}
    and therefore, if $f\in N_x^\sigma$ we conclude
    \begin{align*}
        |\hat{f}(\xi)|
        \leq C e^{-\rho\,\log^{\frac{\sigma}{\sigma-1}}|\xi|}
        \leq C e^{-\rho'g_{\sigma}(\xi)},
    \end{align*}
   for suitable $C,\rho'>0$ and $|\xi|$ large enough. This proves that $f\in \mathcal{E}_{\sigma} ^{\Gamma}$. Inclusion $\mathcal{E}_{\sigma} ^{\Gamma} \subset \mathcal{E}_{\sigma}$ follows directly from the Theorem \ref{Thm:Peli-Viner}.
\end{proof}

\subsection{Illustrations}\label{secIlustracije}
In this section, we present illustrations of the low-pass filter constructed in the proof of Theorem~\ref{MainResult}, together with the corresponding scaling function and wavelet. All figures were generated in \textsc{Matlab} (the code is available at https://sites.google.com/view/goalsproject/publications). The parameters used in \eqref{theta} and \eqref{NasLowPass} are $\sigma = 2$ and $d = \frac{\pi}{12}$.

\begin{figure}[H]
     \centering
     \includegraphics[scale=0.4]{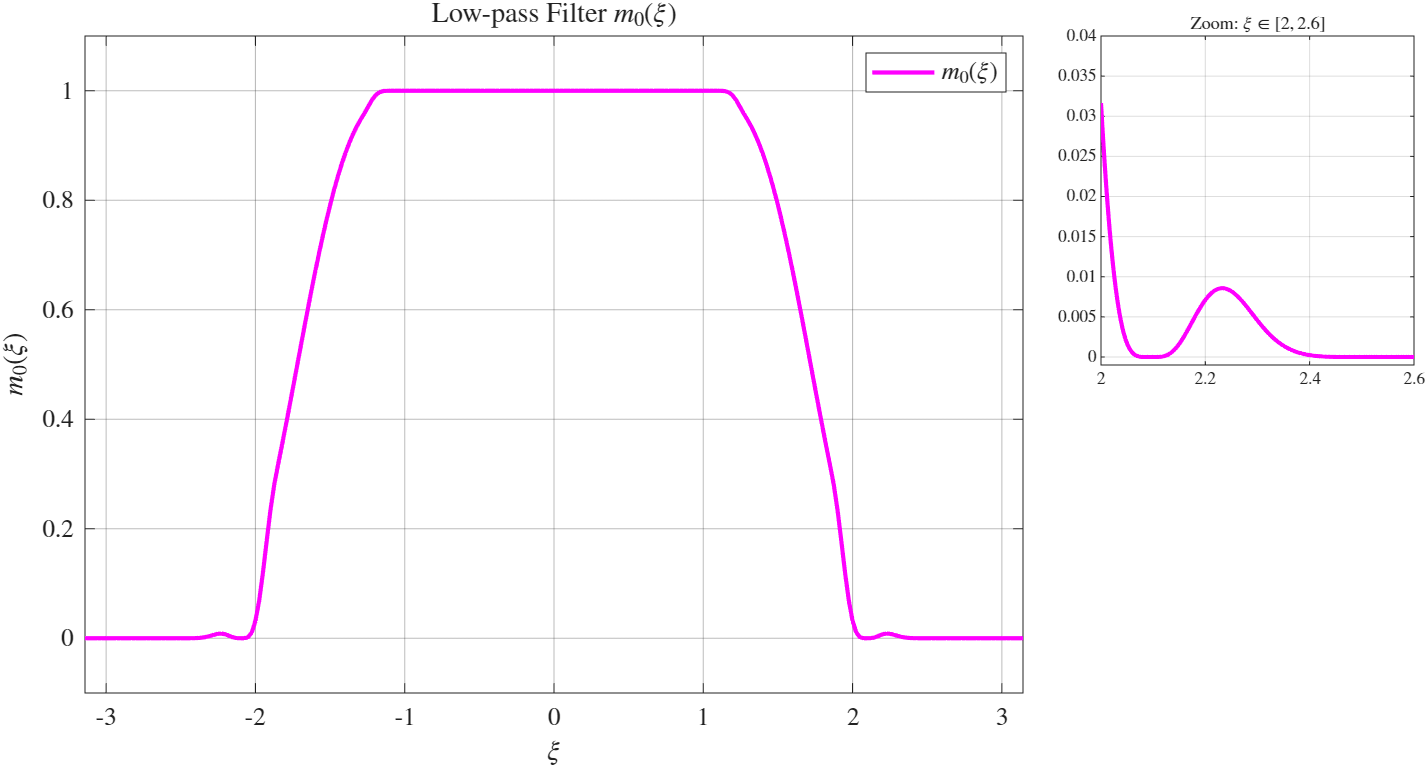}
     \caption{The plot of the low-pass filter \(m_0\) given in \eqref{NasLowPass}. Note that it features a small bump near \(\tfrac{2\pi}{3}\) (see the smaller piece), that enables control of its local decay.
}
\end{figure}

\begin{figure}[H]
     \centering
     \includegraphics[scale=0.4]{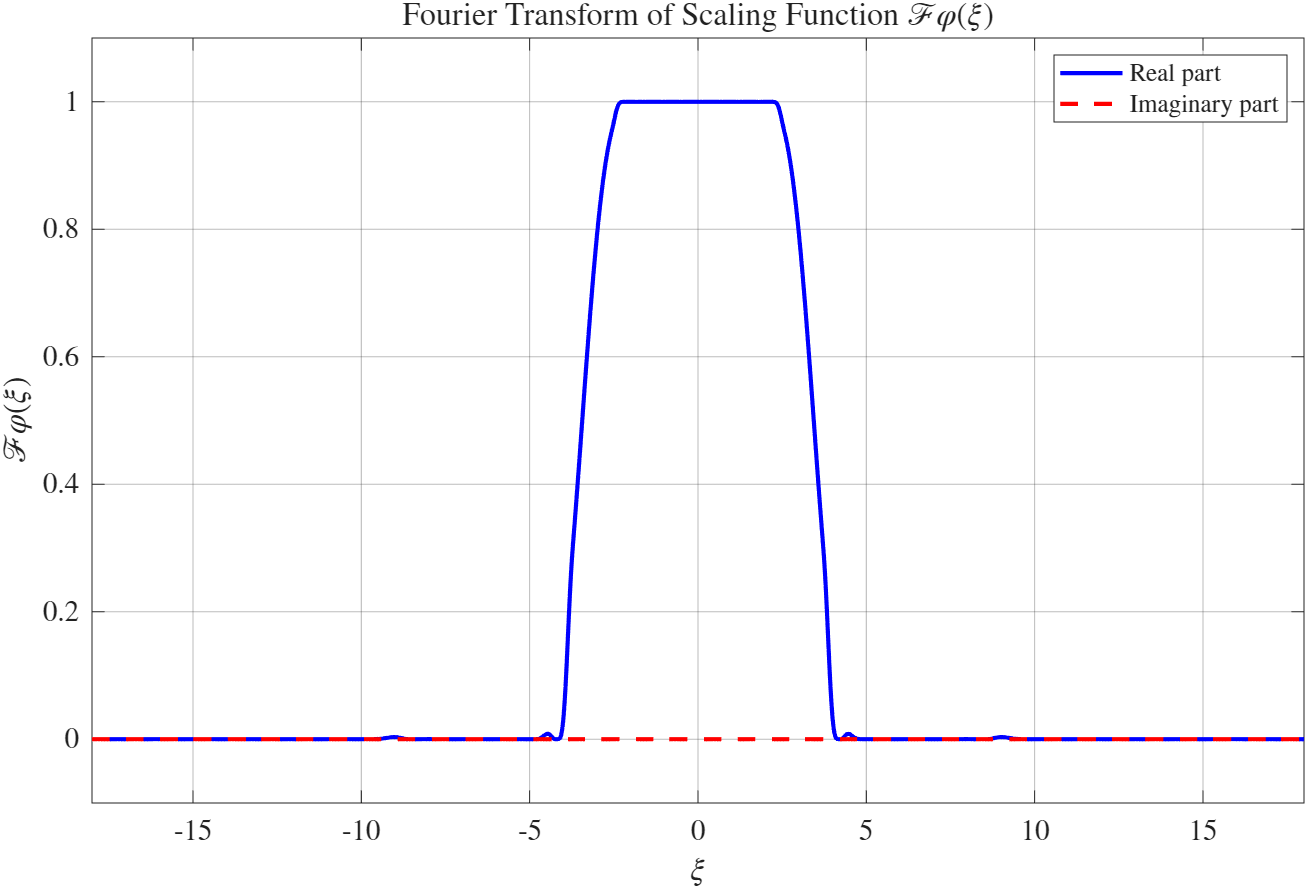}
     \caption{Plot of the Fourier transform of the scaling function; $\hat\varphi$. }
\end{figure}

\begin{figure}[H]
     \centering
     \includegraphics[scale=0.4]{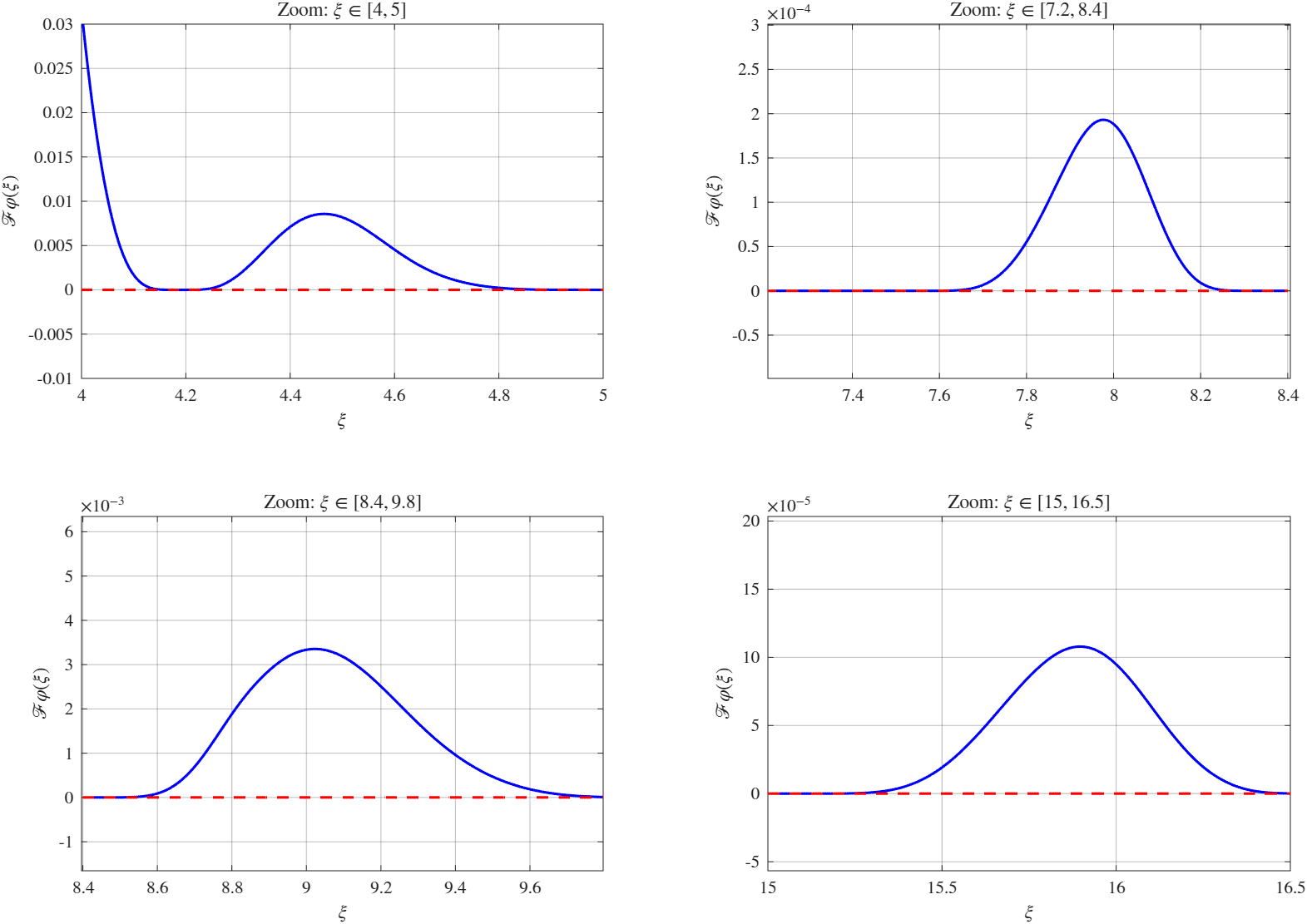}
     \caption{Note that \(\hat{\varphi}\) inherits the bumps of \(m_0\) near the invariant cycle points. Therefore \(\hat{\varphi}\) is not band-limited.
}
\end{figure}

\begin{figure}[H]
     \centering
     \includegraphics[scale=0.4]{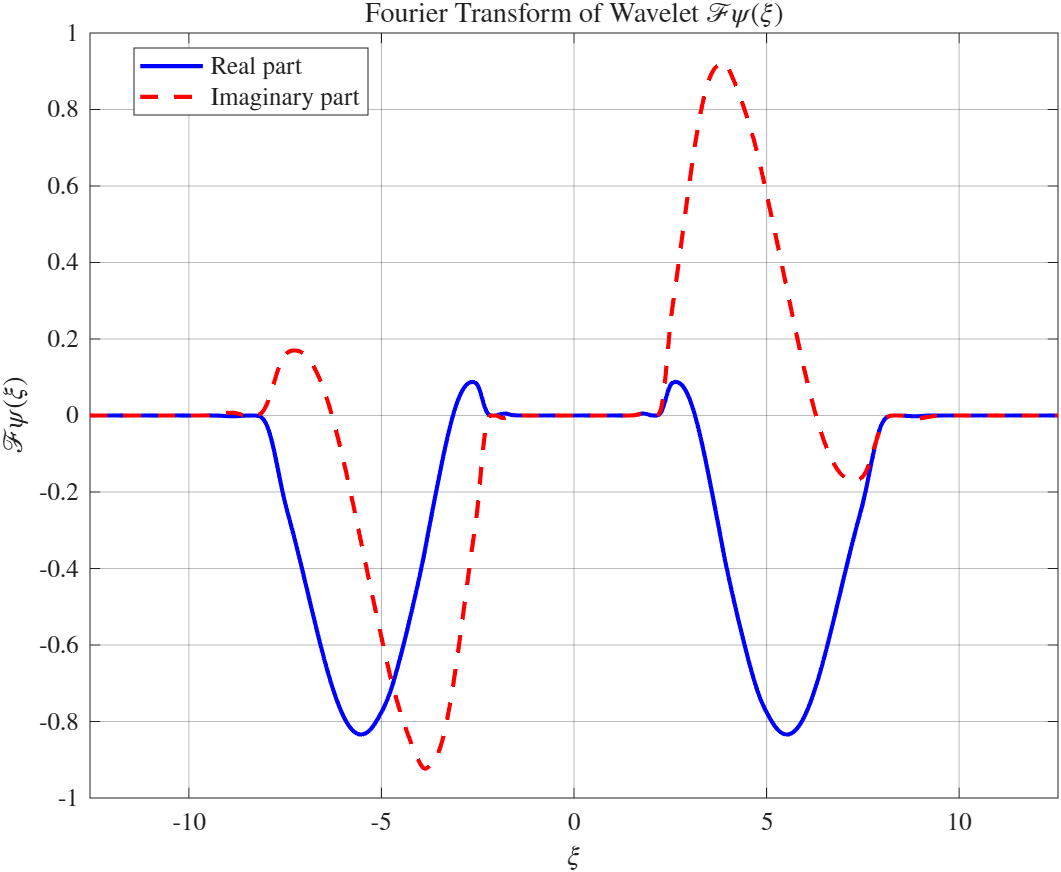}
     \caption{Plot of  Fourier transform of the wavelet; $\hat\psi$}
\end{figure}

\begin{figure}[H]
     \centering
     \includegraphics[scale=0.4]{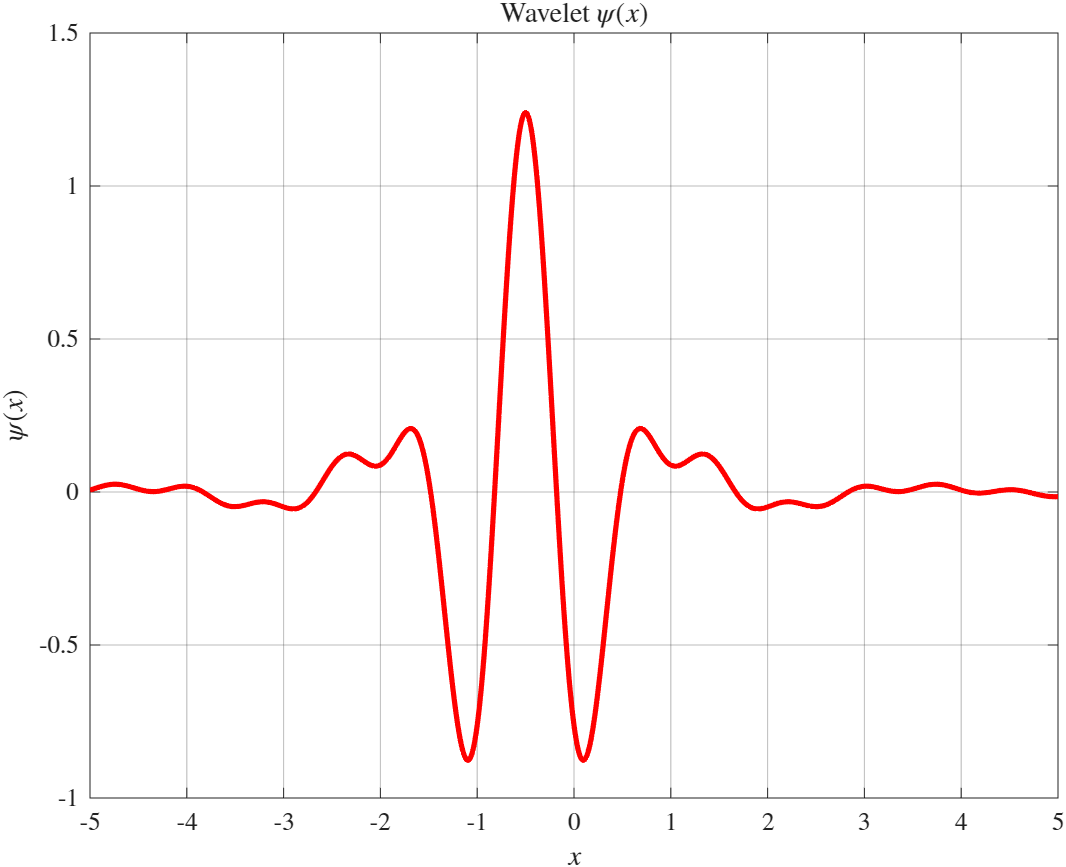}
     \caption{Plot of the wavelet $\psi$}
\end{figure}

\section{Appendix}
\begin{proof}[{\bf Proof of Lemma \ref{LemmaFlat}}]

 
 We prove that $\displaystyle \lim_{x\to 0^+} f_{\rho_{\sigma},\sigma} ^{(j)} \left( x \right) = 0$, where $f_{\rho_{\sigma},\sigma}$ is given in \eqref{PrimerExtended}.

For $x>0$ and $j\in\mathbb N$ we have
 \begin{equation}
\label{ntiIzvodf_sigma}
    f_{\rho_{\sigma},\sigma} ^{(j)} \left(x\right)
    = \left(e^{-\rho_{\sigma} g_{\sigma} \left(\frac{1}{x}\right) }\right)^{(j)}
    = \left( \left( e^{-\rho_{\sigma}\,\cdot\,} \circ \omega_{\sigma} (\,\cdot\,)\circ \ln\left( 1+\frac{1}{\cdot}\right) \right) (x) \right)^{(j)},
\end{equation} where $g_{\sigma}$ and $\omega_{\sigma}$ are given in \eqref{BazicnaFunkcija}.

Note that
\begin{equation}
\label{IzvodiLN}
\left(\ln \left(1+\frac{1}{x}\right)\right)^{(n)} = (-1)^{n-1} (n-1)! \left( \frac{1}{(x+1)^n} - \frac{1}{x^n} \right),\quad x>0,\; n\in\mathbb N.
\end{equation} 

Let us compute the $n$-th derivative of $\omega_{\sigma} (x)$. Using Leibniz product formula we obtain
\begin{align*}
    \omega_{\sigma}^{(n)} (x)
    &= \sum_{k=0}^n \binom{n}{k} \left(x^{\frac{\sigma}{\sigma-1}}\right)^{(n-k)} \left( \frac{1}{W^{\frac{1}{\sigma -1}}(x)} \right)^{(k)} \\
    &= \sum_{k=0}^n \binom{n}{k} \prod_{l=0}^{n-k-1}\left( \frac{\sigma}{\sigma - 1} - l \right) x^{\frac{\sigma}{\sigma-1}-n+k} \left( \frac{1}{W^{\frac{1}{\sigma -1}}(x)} \right)^{(k)},\,x>0,\, n\in\mathbb{N}.
\end{align*}
By Fa\'a di Bruno formula and (W4), for $0\leq k \leq n$ we have 
\begin{align*}
    &\left( \frac{1}{W^{\frac{1}{\sigma -1}}(x)} \right)^{(k)}
    = \sum_{\substack{m_1,\dots , m_k\in\mathbb{N} \\ m_1+2m_2+\dots+km_k=k}} \frac{k! (-1)^{m_1+m_2+\dots+m_k} }{m_1 ! m_2! \dots m_k!} \times \\
    &\times \prod_{j=0}^{m_1+\dots+m_k-1} \left(\frac{1}{\sigma -1} + j\right)
     \frac{1}{W(x)^{\frac{1}{\sigma -1} + m_1+\dots + m_k}}
     \prod_{s=1}^{k} \frac{1}{s!^{m_s}} \left( \frac{W^s (x) p_{s}(W(x)) }{ x^s (1+W(x))^{2s-1}} \right)^{m_s}\\
    &= \frac{1}{x^k} \frac{W(x)^{k-\frac{1}{\sigma -1}}}{(1+W(x))^{2k}} \sum_{\substack{m_1,\dots , m_k\in\mathbb{N} \\ m_1+2m_2+\dots+km_k=k}} A_k (m_1,\dots ,m_k) \left( \frac{1+W(x)}{W(x)} \right)^{m_1+\dots+m_k} \times \\
    & \times \prod_{s=1}^{k} (p_s (W(x)))^{m_s},\quad x>0,
\end{align*}
where
$$A_k (m_1,\dots ,m_k) = \frac{(-1)^{m_1+m_2+\dots+m_k} k!}{m_1 ! m_2! 2!^{m_2}  \dots m_k!k!^{m_k}} \prod_{j=0}^{m_1+\dots+m_k-1} \left(\frac{1}{\sigma - 1} + j\right).$$
Therefore,
\begin{align}
\label{IzvodOmega}
    &\omega_{\sigma}^{(n)} (x) = \sum_{k=0}^n a_{k,n}(\sigma) \, x^{\frac{\sigma}{\sigma-1}-n+k} \frac{1}{x^k} \frac{W^{k-\frac{1}{\sigma-1} }(x)}{(1+W(x))^{2k}} \nonumber\\
    & \quad \sum_{\substack{m_1,\dots , m_k\in\mathbb{N} \\ m_1+2m_2+\dots+km_k=k}} A_k (m_1,\dots ,m_k) \left( \frac{1+W(x)}{W(x)} \right)^{m_1+\dots+m_k}\nonumber 
      \prod_{s=1}^{k} (p_s (W(x)))^{m_s}\nonumber\\
    &= x^{\frac{\sigma}{\sigma-1}-n} \sum_{k=0}^n \sum_{\substack{m_1,\dots , m_k\in\mathbb{N} \\ m_1+2m_2+\dots+km_k=k}} a_{k,n}(\sigma) A_k (m_1,\dots ,m_k) \, \frac{W^{k-\frac{1}{\sigma-1} }(x)}{(1+W(x))^{2k}} \times \nonumber\\
    & \quad \times \left( \frac{1+W(x)}{W(x)} \right)^{m_1+\dots+m_k} \prod_{s=1}^{k} (p_s (W(x)))^{m_s},\quad x>0,\;n\in \mathbb N,
\end{align}
where $a_{k,n}(\sigma) = \binom{n}{k} \prod_{l=0}^{n-k-1}\left( \frac{\sigma}{\sigma - 1} - l \right)$. 

To avoid further technicalities, we observe that \eqref{IzvodiLN},  \eqref{IzvodOmega}, and successive application of the Fa\'a di Bruno formula, imply that \eqref{ntiIzvodf_sigma} can be written as (finite) sum of terms
\begin{align}
\label{fsigmaT}
\alpha_{\sigma} e^{-\rho_{\sigma} g_{\sigma} \left(\frac{1}{x}\right)} \left(\ln \left( 1+\frac{1}{x} \right)\right)^{a_\sigma} \frac{ W^{b_{\sigma}} \left(\ln \left( 1+\frac{1}{x} \right)\right) }{ \left( 1+W\left(\ln \left( 1+\frac{1}{x} \right)\right) \right)^c } \frac{1}{x^{d'} (x+1)^{d''} } , \quad x>0,
\end{align}
for suitable $\alpha_{\sigma}, a_{\sigma} , b_{\sigma} \in \mathbb{R} , \, c, d' , d'' \geq 0$, where we also recall that $p_s (W(x))$ appearing in \eqref{IzvodOmega} is polynomial with respect to $W(x)$.

Setting $\displaystyle s=\ln \left( 1+\frac{1}{x} \right)$ in \eqref{fsigmaT} and using \eqref{PosledicaLambert1.5}, we obtain
\begin{align}
\label{Sumandi} 
e^{-\rho_{\sigma}\omega_{\sigma} (s)} s^{a_{\sigma}} \frac{W^{b_{\sigma}} (s)}{(1+W(s))^c} \frac{(e^s-1)^{d'+d''}}{(e^s)^{d''}}
    \sim  e^{-\rho_{\sigma}\omega_{\sigma} (s)} s^{a_{\sigma}} \ln^{b_{\sigma}-c}(s) e^{d's},    
\end{align} when $s\to \infty$.

Note that there exist constants $\beta , d_0 > 0$ such that
 $$|s^{a_{\sigma}} \ln^{b_{\sigma}-c}(s) e^{d's}| \leq \beta e^{d_0 s},\quad s>0,$$ hence Lemma \ref{spanskalema} $b)$ and $\eqref{Sumandi}$ imply that each term of the form \eqref{fsigmaT} tends to 0 when $s\to+\infty$. Finally, this proves $\lim_{x\to 0^+} f_{\rho_{\sigma},\sigma} ^{(j)} ( x ) = 0$.
 \end{proof}

 \section{Acknowledgments}

The authors thank N. Teofanov for helpful discussions.
This research was supported by the Science Fund of the Republic of
Serbia, $\#$GRANT No. 2727, {\it Global and local analysis of operators and
distributions} - GOALS.  The authors were also supported by the Ministry of Science, Technological Development, and Innovation of the Republic of Serbia - the second and third authors by Grants No. 451-03-137/2025-03/200125 and 451-03-136/2025-03/200125, and the first author by Grant No. 451-03-136/2025-03/200156.

\end{document}